\documentclass[final,3p,times]{elsarticle}

\usepackage[T1]{fontenc}
\usepackage{times}
\usepackage{epsfig,graphicx,tabularx,color}
\usepackage[cp850]{inputenc}
\usepackage[english]{babel}
\usepackage{amsmath,amssymb,amsfonts,amsthm,mathrsfs,comment,mathdots,multirow,tikz}
\usepackage{rotating, longtable, relsize}
\usepackage[f]{esvect}
\usepackage{times}
\usepackage{fancyhdr}
\usepackage{lipsum}
\usepackage{soul}

\usetikzlibrary{calc,3d}

\newtheorem{theorem}{Theorem}[section]

\newtheorem{lemma}[theorem]{Lemma}

\newtheorem{definition}[theorem]{Definition}
\newtheorem{proposition}[theorem]{Proposition}

\theoremstyle{definition}

\def\XXint#1#2#3{{\setbox0=\hbox{$#1{#2#3}{\int}$}
     \vcenter{\hbox{$#2#3$}}\kern-.5\wd0}}

\def\ud{{\rm\,d}}

\def\R{\mathbb{R}}
\def\Sph{\mathbb{S}}

\def\LL{\mathcal{L}}

\def\NN{\mathcal{N}}
\def\OO{\mathcal{O}}

\def\xb{\mathbf{x}}
\def\yb{\mathbf{y}}
\def\zb{\mathbf{z}}

\newcommand{\dsp}{\displaystyle}

\def\pr(#1){\left({#1}\right)}
\def\br[#1]{\left[{#1}\right]}

\def\abs#1{\left|{#1}\right|}
\def\norm#1{\left\|{#1}\right\|}
\def\conj#1{\overline{#1}}

\def\ii{{\rm i}}

\setulcolor{green}
\soulregister\cite7
\soulregister\ref7
\soulregister\pageref7

\DeclareRobustCommand{\rchi}{{\mathpalette\irchi\relax}}
\newcommand{\irchi}[2]{\raisebox{\depth}{$#1\chi$}} 

\journal{Journal of Computational Physics}

\begin{document}

\begin{frontmatter}



\title{A spectral method for nonlocal diffusion operators on the sphere}


\author[RMS]{Richard Mika\"el Slevinsky\corref{cor1}}
\ead{Richard.Slevinsky@umanitoba.ca}
\ead[url]{https://home.cc.umanitoba.ca/~slevinrm/}

\author[HMQD]{Hadrien Montanelli}
\ead{hlm2137@columbia.edu}

\author[HMQD]{Qiang Du}
\ead{qd2125@columbia.edu}

\cortext[cor1]{Corresponding author}
\address[RMS]{Department of Mathematics, University of Manitoba, Winnipeg, Canada}
\address[HMQD]{Department of Applied Physics and Applied Mathematics, Columbia University, 
New York, United States}

\begin{abstract}
We present algorithms for solving spatially nonlocal diffusion models on the unit sphere with spectral accuracy in space. Our algorithms are based on the diagonalizability of nonlocal diffusion operators in the basis of spherical harmonics, the computation of their eigenvalues to high relative accuracy using quadrature and asymptotic formulas, and a fast spherical harmonic transform. These techniques also lead to an efficient implementation of high-order exponential integrators for time-dependent models. We apply our method to the nonlocal Poisson, Allen--Cahn and Brusselator equations.
\end{abstract}


\end{frontmatter}

\section{Introduction}

Nonlocal models have been extensively studied in many fields such as materials science,
thermodynamics, fluid dynamics, fracture mechanics, biology and image analysis~\cite{bates1999, bobaru2010, du2017a, gilboa2008, kao2010, silling2000}.
Many of these models can be conveniently formulated using nonlocal integral operators generalizing the standard differential operators of vector calculus~\cite{du2012, du2013}. 
In this paper, we propose a fast spectral method for computing solutions of nonlocal models of the form
\begin{equation}
\label{PDE}
u_t = \epsilon^2\LL_\delta u + \NN(u), \quad u(t=0,\xb) = u_0(\xb), \quad \epsilon>0,
\end{equation}

\noindent where $u(t,\xb)$ is a function of time $t\ge0$ and position $\xb$ on the unit sphere $\mathbb{S}^2\subset\mathbb{R}^3$,
$\NN$ is a nonlinear operator with constant coefficients (e.g., $\NN(u)=u-u^3$), and $\LL_\delta$ is a nonlocal \textit{Laplace--Beltrami operator},
\begin{equation}
\LL_\delta u({\bf x}) = \int_{\Sph^2} \rho_\delta(\abs{{\bf x}-{\bf y}})\left[ u({\bf y}) - u({\bf x})\right]\ud\Omega(\yb).
\label{operator}
\end{equation}

\noindent In the definition \eqref{operator} above, $\vert\xb -\yb\vert$ is the Euclidean distance between $\xb$ and $\yb$ in $\R^3$,
$\ud\Omega(\yb)$ denotes the standard measure on $\Sph^2$ and $\rho_\delta$ is a suitably defined nonlocal kernel with horizon $0<\delta\le2$, 
which determines the range of interactions. 
$\LL_\delta$ is also called a nonlocal diffusion operator on the sphere.
The function $u$ can be real or complex and the equation \eqref{PDE} can be a single equation as well as a system of equations.

There has been substantial work on the numerical approximation of the equation \eqref{PDE} in Euclidean domains~\cite{du2017b, du2018},
including a spectral method for nonlocal diffusion operators defined over a periodic cell in $\R^d$ ($d\leq 3$) \cite{du2017a, du2016}.
However, no study has been attempted so far to investigate similar discretizations on the sphere. It is of practical interests 
to study the extension to non-Euclidean geometries with the sphere being a representative example, e.g., for the modelling of anomalous diffusion, pattern formation and image analysis.

For problems defined in spherical geometries there are several methods to discretize the spatial part of the local 
partial differential equation (PDE) version of the equation \eqref{PDE} with spectral accuracy.
The most popular methods include spherical harmonics~\cite{atkinson2012} and the double Fourier sphere (DFS) method~\cite{merilees1973, orszag1973}, 
recently revisited by Montanelli and Nakatsukasa~\cite{montanelli2017phd, montanelli2017b} and Townsend et al.~\cite{townsend2016}.
On the one hand, spherical harmonics are the natural spectral basis for solving local and nonlocal equations on the sphere since they diagonalize both the local and
nonlocal Laplace--Beltrami operators, as we will show in Section~2.
On the other hand, the DFS method is efficient because it allows one to use two-dimensional fast Fourier transforms (FFTs), which leads to $\OO(n^2\log n)$ complexity per time-step 
when the spatial part of \eqref{PDE} is discretized at $\OO(n^2)$ points. 

For spherical harmonics, the picture was quite different before the 2000s. 
Since a conventional synthesis and analysis of spherical harmonic expansions on tensor-product grids costs $\OO(n^3)$, the cost per time-step was not competitive. 
The picture became more similar from 2006 when Tygert, first with Rokhlin~\cite{rokhlin2006} then independently \cite{tygert2008, tygert2010}, developed asymptotically optimal spherical harmonic transforms with $\OO(n^2\log n)$ run times. 
The pre-computation was also asymptotically optimal, though this is anticipated only for absurdly high degrees.
More recently, Slevinsky proposed two new fast spherical harmonic transforms, first with backward stability, $\OO(n^2\log^2 n)$ run-time and $\OO(n^3\log n)$ pre-computation costs \cite{slevinsky2017a},
then with $\OO(n^2\log^2 n)$ run-time and $\OO(n^{\frac{3}{2}}\log n)$ pre-computation complexities \cite{slevinsky2017b}.
Equipped with the first fast transform of Slevinsky, we present a spectral method for solving nonlocal diffusion models on the sphere based on eigenvalues and spherical harmonics, 
which we combine with high-order exponential integrators like ETDRK4~\cite{cox2002} for time-dependent models.

The paper is structured as follows. 
In Section 2, we prove that the nonlocal Laplace--Beltrami operator on the sphere is diagonalized in the basis of 
spherical harmonics and derive a closed-form expression for its eigenvalues, which we compute to high relative accuracy in Section 3.
We then solve the nonlocal Poisson equation in Section 4 and nonlocal time-dependent equations in Section 5.

\section{Eigendata of the nonlocal Laplace--Beltrami operator on the sphere}

Let $\xb\in\Sph^2$ be a point on the sphere parameterized by the angles $(\theta,\varphi)$, where $\theta\in[0,\pi]$ is the colatitude and $\varphi\in[0,2\pi)$ is the longitude,
and let $\ud\Omega = \sin\theta\ud\theta\ud\varphi$ be the measure generated by the solid angle $\Omega$ subtended by a spherical cap.
The spherical harmonics as given by
\begin{equation}\label{eq:sphericalharmonics}
Y_\ell^m(\xb) = Y_\ell^m(\theta,\varphi) = \dfrac{e^{\ii m\varphi}}{\sqrt{2\pi}} \underbrace{\ii^{m+|m|}\sqrt{(\ell+\tfrac{1}{2})\dfrac{(\ell-m)!}{(\ell+m)!}} P_\ell^m(\cos\theta)}_{\tilde{P}_\ell^m(\cos\theta)},
\quad \ell\geq0,\; -\ell\le m\le \ell,
\end{equation}

\noindent where the notation $\tilde{P}_\ell^m$ for the associated Legendre polynomials is used to denote orthonormality 
for fixed $m$ in the sense of $L^2([-1,1])$, are eigenfunctions with eigenvalues $-\ell(\ell+1)$ of the local Laplace--Beltrami 
operator $\LL_0$ defined by
\begin{equation}
\LL_0 u(\theta,\varphi) = u_{\theta\theta} + \frac{\cos\theta}{\sin\theta}u_\theta + \frac{u_{\varphi\varphi}}{\sin^2\theta},
\end{equation}

\noindent that is,
\begin{equation}
\LL_0 Y_\ell^m(\theta,\varphi) = -\ell(\ell+1)Y_\ell^m(\theta,\varphi).
\end{equation}

Spherical harmonics $Y_\ell^m$ and Legendre polynomials $P_\ell$ satisfy the addition theorem~\cite[Eq.~(2.27)]{atkinson2012}
\begin{equation}
P_\ell(\xb\cdot\yb) = \dfrac{4\pi}{2\ell+1}\sum_{m=-\ell}^{+\ell} Y_\ell^m({\bf x})\conj{Y_\ell^m({\bf y})},
\end{equation}

\noindent which gives us the integral representation~\cite[Eq.~(2.33)]{atkinson2012}
\begin{equation}
Y_\ell^m(\xb) = \dfrac{2\ell+1}{4\pi}\int_{\Sph^2} P_\ell(\xb\cdot\yb) Y_\ell^m(\yb) \ud\Omega(\yb).
\end{equation}

\noindent Furthermore, for $f\in L^1([-1,1])$, the Funk--Hecke formula states that~\cite[Th.~2.22]{atkinson2012}
\begin{equation}
\int_{\Sph^2} f(\xb\cdot\zb) P_\ell(\yb\cdot\zb)\ud\Omega(\zb) = 2\pi P_\ell(\xb\cdot\yb)\beta_\ell, \quad{\rm where}\quad \beta_\ell=\int_{-1}^1P_\ell(t)f(t)\ud t.
\label{FunkHecke}
\end{equation}

\noindent From the Funk--Hecke formula \eqref{FunkHecke} we obtain the following proposition, which will be useful for our particular choice for the kernel 
$\rho_\delta$ in \eqref{operator}. The proof is given in Appendix~\ref{appendix:GFH}.

\begin{proposition}[Generalized Funk--Hecke formula]\label{proposition:GFH}
For a function $f$ with $t\mapsto(1-t)f(t)\in L^1([-1,1])$, we have
\begin{equation}
\int_{\Sph^2} f(\xb\cdot\zb)[P_\ell(\yb\cdot\zb)-P_\ell(\xb\cdot\yb)]\ud\Omega(\zb) 
= 2\pi P_\ell(\xb\cdot\yb)\gamma_\ell, \quad \gamma_\ell=\int_{-1}^1[P_\ell(t)-1]f(t)\ud t.
\label{GeneralizedFunkHecke}
\end{equation}
\end{proposition}

Let us consider now the nonlocal operator \eqref{operator}. We will set
\begin{equation}
\tilde{\rho}_\delta({\bf x}\cdot{\bf y}) = \rho_\delta(\abs{{\bf x}-{\bf y}}) = \rho_\delta(\sqrt{2(1-{\bf x}\cdot{\bf y})})
= \rho_\delta(\sqrt{2(1-t)})
\end{equation}

\noindent for brevity, where we used the fact that the Euclidean distance between two points $\xb,\yb\in\Sph^2$ is
\begin{equation}
0 \le \abs{\xb-\yb} = \sqrt{2(1-\xb\cdot\yb)} \le 2.
\end{equation}

\noindent Let us assume that our kernel satisfies $t\mapsto(1-t)\rho_\delta(\sqrt{2(1-t)})\in L^1([-1,1])$ so we can use the generalized Funk--Hecke 
formula \eqref{GeneralizedFunkHecke}---we will come back to this later.
If we could find eigenfunctions of the nonlocal operator, then it may be possible to compute eigenvalues as well.
For our nonlocal operator, we find
\begin{equation}
\begin{array}{ll}
\LL_\delta Y_\ell^m(\xb) & \dsp = \int_{\Sph^2} \tilde{\rho}_\delta(\xb\cdot\zb) \left[ Y_\ell^m(\zb) - Y_\ell^m(\xb)\right]\ud\Omega(\zb),\\[10pt]
& \dsp = \dfrac{2\ell+1}{4\pi}\int_{\Sph^2} \tilde{\rho}_\delta(\xb\cdot\zb)\int_{\Sph^2} Y_\ell^m(\yb) \left[P_\ell(\yb\cdot\zb) - P_\ell(\xb\cdot\yb)\right] \ud\Omega(\yb)\ud\Omega(\zb),\\[10pt]
& \dsp = \dfrac{2\ell+1}{4\pi}\int_{\Sph^2} Y_\ell^m(\yb) \int_{\Sph^2} \tilde{\rho}_\delta(\xb\cdot\zb) \left[ P_\ell(\yb\cdot\zb) - P_\ell(\xb\cdot\yb) \right] \ud\Omega(\zb)\ud\Omega(\yb),\\[10pt]
& \dsp = 2\pi\gamma_\ell\dfrac{2\ell+1}{4\pi}\int_{\Sph^2} Y_\ell^m(\yb)P_\ell(\xb\cdot\yb)\ud\Omega(\yb), \\[10pt]
& \dsp = 2\pi\gamma_\ell Y_\ell^m(\xb).
\end{array}
\end{equation}

\noindent Therefore the spherical harmonics $Y_\ell^m$ are eigenfunctions of the nonlocal Laplace--Beltrami operator \eqref{operator} with eigenvalues
\begin{equation}
\lambda_\delta(\ell) = 2\pi\gamma_\ell = 2\pi\int_{-1}^1\big[P_\ell(t)-1\big]\rho_\delta(\sqrt{2(1-t)})\ud t.
\end{equation}

\noindent In the following we will focus on the weakly singular kernel
\begin{equation}
\rho_\delta(\sqrt{2(1-t)}) 
= \dfrac{(1+\alpha)2^{1+\alpha}}{\pi\delta^{2+2\alpha}(1-t)^{1-\alpha}}\rchi_{[0,\delta]}(\sqrt{2(1-t)}),
\quad -1<\alpha<1,\; 0<\delta\leq2,
\label{kernel}
\end{equation}

\noindent where $\rchi_{[0,\delta]}(\cdot$ is the indicator function, which results in eigenvalues
\begin{equation}
\lambda_\delta(\ell) = 2\pi\int_d^1 \big[P_\ell(t)-1\big]\rho_\delta(\sqrt{2(1-t)}) \ud t, \quad d = 1-\delta^2/2.
\label{eigenvalues}
\end{equation}

\noindent Note that our kernel satisfies $t\mapsto(1-t)\rho_\delta(\sqrt{2(1-t)})\in L^1([-1,1])$, which justifies the use
of the generalized Funk--Hecke formula \eqref{GeneralizedFunkHecke} in the derivation of the eigenvalues.
The indicator function in $\eqref{kernel}$ is to impose the limit on the range of interactions, 
and the constants in $\eqref{kernel}$ ensure that for every $\alpha\in(-1,1)$, 
$\lambda_\delta\to\lambda_0$ and $\LL_\delta \rightarrow \LL_0$ strongly as $\delta\to0$ in a suitable operator norm.
This statement is proved in Appendix~\ref{appendix:strongconvergence} for an induced operator norm mapping between Sobolev spaces.
Note also that this type of kernel is standard for nonlocal models and is the analogue in spherical coordinates of the kernel used for Euclidean domains in \cite{du2017a}.
Finally, we expect that the spectral discretization of our nonlocal model \eqref{PDE} on the sphere shares the asymptotic compatibility~\cite{du2014}, 
demonstrated for flat geometries in~\cite{du2016}. Asymptotic compatibility means that the local limit can be preserved at the discrete level and the convergence is uniform with respect to $\delta$ as $\delta\to 0$.

\section{Numerical evaluation of the spectrum}

A key ingredient of a spectral method is a fast and accurate computation of the spectrum of the underlying 
operators~\cite{trefethen2000, canuto2007, shen2011}.
In this section we present our numerical method for computing the eigenvalues \eqref{eigenvalues}.

\subsection{Methodology}

We rescale the integral in \eqref{eigenvalues} by
\begin{equation}
t = \frac{1+d}{2} + \frac{1-d}{2}x,
\end{equation}

\noindent such that $t(x) : [-1,1] \to [d,1]$, which yields
\begin{equation}
\lambda_\delta(\ell) = \pi(1-d) \int_{-1}^1 \left[P_\ell\left(t(x)\right)-1\right]\rho_\delta(\sqrt{(1-d)(1-x)})\ud x.
\end{equation}

\noindent Now, $P_\ell$ is a degree-$\ell$ polynomial and
\begin{equation}
\rho_\delta(\sqrt{(1-d)(1-x)}) = \dfrac{(1+\alpha)2^{1+\alpha}}{\pi\delta^{2+2\alpha}\big[(\frac{1-d}{2})(1-x)\big]^{1-\alpha}}.
\end{equation}

\noindent Therefore,
\begin{equation}
\lambda_\delta(\ell) = \frac{(1+\alpha)2^{2+\alpha}}{\delta^{2+2\alpha}}\left(\frac{1-d}{2}\right)^\alpha \int_{-1}^1 \left[P_\ell\left(t(x)\right)-1\right](1-x)^{\alpha-1}\ud x.
\end{equation}

\noindent This integrand has an algebraic singularity of order $\alpha>-1$ at the right endpoint. To see this, we note that
\begin{equation}
\begin{array}{ll}
\dsp P_\ell\left(\frac{1+d}{2} + \frac{1-d}{2}x\right)-1 & \dsp \sim P_\ell(1)-1 - (1-x)\left.P_\ell'\left(\frac{1+d}{2} + \frac{1-d}{2}x\right)\right|_{x=1}\\[10pt]
& \dsp = -(1-x)\dfrac{(1-d)}{2}\dfrac{\ell(\ell+1)}{2},\quad{\rm as}\quad x\to1^-.
\end{array}
\end{equation}

\noindent Therefore, the final form to integrate is
\begin{equation}
\lambda_\delta(\ell) = \frac{(1+\alpha)2^{2-\alpha}}{\delta^2}\int_{-1}^1 \dfrac{P_\ell\left[1 -\frac{\delta^2}{2}\left(\frac{1-x}{2}\right)\right]-1}{1-x}(1-x)^\alpha\ud x.
\label{eigenvalues2}
\end{equation}

The hybrid Taylor/RK4 algorithm of Du and Yang~\cite{du2017a} requires an ordinary differential equation in terms of a continuously defined eigenvalue with respect to the index $\ell$. 
The difficulty of the spherical eigenvalue formulation is that the oscillations appear due to the degree $\ell$ of the Legendre polynomials, 
and derivatives with respect to $\ell$ unnecessarily complicate the matter.
However, the advantage of the spherical eigenvalue formulation is that, compared to the eigenvalue computation on the two-dimensional torus, 
there are only $n+1$ distinct eigenvalues to calculate for the $(n+1)^2$ spherical harmonics of degree $\le n$ due to the degeneracy of the spherical harmonics of a particular degree.

The algorithm we propose is to integrate using a modified \textit{Clenshaw--Curtis quadrature} rule. 
The computational complexity is $\OO(\ell^2 + \ell\log \ell) = \OO(\ell^2)$ per eigenvalue $\lambda_\delta(\ell)$, 
since the integrand requires the numerical evaluation of a degree-$\ell$ Legendre polynomial at $\ell+1$ points. 
However, when $\ell$ is sufficiently large, the integrand evaluation can be replaced by \textit{Szeg{\H o}'s asymptotic formula}~\cite{szego1934}, 
which reduces the complexity of pointwise evaluation to $\OO(1)$ and the overall complexity to $\OO(\ell + \ell\log \ell) = \OO(\ell\log \ell)$ per eigenvalue.

\paragraph{Clenshaw--Curtis quadrature}
Clenshaw--Curtis quadrature is a quadrature rule~\cite{sommariva2013, waldvogel2006} whose nodes are the $\ell+1$ Chebyshev--Lobatto points $x_k = \cos(k\pi/\ell)$. 
Given a continuous weight function $w(x)$ and $\mathbb{P}_\ell$, the space of algebraic polynomials of degree at most $\ell$, the quadrature weights $w_k$ are determined by
\begin{equation}
\int_{-1}^1 f(x)w(x){\rm\,d}x = \sum_{k=0}^\ell w_k f(x_k),\quad\forall f\in\mathbb{P}_\ell.
\end{equation}

\noindent With the modified Chebyshev moments of the weight function
\begin{equation}
\mu_k = \int_{-1}^1 T_k(x) w(x){\rm\,d}x,\quad 0\leq k\leq \ell,
\end{equation}

\noindent the quadrature weights can be determined via the formula
\begin{equation}\label{eq:CCweights}
w_j =
\dfrac{1-\tfrac{1}{2}(\delta_{0,j}+\delta_{\ell,j})}{\ell}\left[\displaystyle \mu_0 + (-1)^j\mu_\ell + 2\sum_{k=1}^{\ell-1} \mu_k\cos(\pi jk/\ell)\right],
\end{equation}

\noindent where $\delta_{k,j}$ is the Kronecker delta~\cite{abramowitz1965}. 
Due to this representation, the ${\cal O}(\ell\log \ell)$ computation of the weights from modified Chebyshev moments is achieved via a diagonally scaled discrete cosine transform.

For the Jacobi weight $w(x)=(1-x)^\alpha(1+x)^\beta$, where $\alpha,\beta>-1$ (which covers our case), the modified Chebyshev moments are known explicitly~\cite{piessens1987}
\begin{equation}
\mu_\ell^{(\alpha,\beta)} = \int_{-1}^1 T_\ell(x) (1-x)^\alpha(1+x)^\beta \ud x = 2^{\alpha+\beta+1}{\rm B}(\alpha+1,\beta+1)\,_3F_2\left( 
\begin{array}{l}
\ell,-\ell,\alpha +1\\
\frac{1}{2}, \alpha+\beta+2
\end{array};
1\right),
\end{equation}

\noindent where $\,_3F_2$ is a generalized hypergeometric function~\cite[\S 16.2.1]{olver2010} and ${\rm B}$ is the beta function~\cite[\S 5.12.1]{olver2010}. 
Using Sister Celine's technique~\cite[\S 127]{rainville1960} or induction~\cite{xiang2014}, a recurrence relation can be derived for the modified moments
\begin{align}
\mu_0^{(\alpha,\beta)} = 2^{\alpha+\beta+1}{\rm B}(\alpha+1,\beta+1),\quad \mu_1^{(\alpha,\beta)} = \dfrac{\beta-\alpha}{\alpha+\beta+2}\mu_0^{(\alpha,\beta)},&\\
(\alpha+\beta+\ell+2)\mu_{\ell+1}^{(\alpha,\beta)} + 2(\alpha-\beta) \mu_\ell^{(\alpha,\beta)} + (\alpha+\beta-\ell+2)\mu_{\ell-1}^{(\alpha,\beta)} = 0,&\quad{\rm for}\quad \ell>0.
\end{align}

Once we have computed the quadrature weights from the modified Chebyshev moments $\mu_\ell^{(\alpha,\beta)}$, all that remains is to evaluate Legendre polynomials at Chebyshev--Lobatto points, which can be done at linear cost per point using the three-term recurrence relation
\begin{equation}
(\ell+1)P_{\ell+1}(t) = (2\ell+1)tP_\ell(t) - \ell P_{\ell-1}(t),
\end{equation}

\noindent or at $\OO(1)$ cost when $\ell$ is large enough using Szeg{\H o}'s asymptotic formula.

\paragraph{Szeg{\H o}'s asymptotic formula}
Szeg{\H o}'s asymptotic formula~\cite{szego1934} for Legendre polynomials is uniformly convergent on $0\le \theta < 2(\sqrt{2}-1)\pi$, though special care must be taken at $\theta=0$.
It has been used by Bogaert~\cite{bogaert2014} to develop iteration-free computation of Gauss--Legendre quadrature nodes and weights. 
The formula is an infinite series in terms of cylindrical Bessel functions
\begin{equation}
P_\ell(\cos\theta) = \sqrt{\dfrac{\theta}{\sin\theta}}\sum_{\nu=0}^\infty \dfrac{a_\nu(\theta) J_\nu((\ell+\tfrac{1}{2})\theta)}{(\ell+\tfrac{1}{2})^\nu},
\end{equation}

\noindent where the first few coefficient functions $a_\nu(\theta)$ are given by
\begin{equation}
\begin{array}{ll}
a_0(\theta) & = 1,\\[10pt]
a_1(\theta) & \dsp = \dfrac{1}{8}\dfrac{\theta\cos\theta-\sin\theta}{\theta\sin\theta},\\[10pt]
a_2(\theta) & \dsp = \dfrac{1}{128}\dfrac{6\theta\sin\theta\cos\theta-15\sin^2\theta+\theta^2(9-\sin^2\theta)}{\theta^2\sin^2\theta},\\[10pt]
a_3(\theta) & \dsp = \dfrac{5}{1024}\dfrac{((\theta^3+21\theta)\sin^2\theta+15\theta^3)\cos\theta-((3\theta^2+63)\sin^2\theta-27\theta^2)\sin\theta}{\theta^3\sin^3\theta},
\end{array}
\end{equation}

\noindent and more can be computed by computer algebra systems.
In our numerical results, it is clear that even the first four terms provide sufficiently accurate pointwise evaluation for the numerical evaluation of the spectrum to high relative accuracy for reasonably low degrees. 
While it appears that each term in Szeg{\H o}'s asymptotic formula requires the numerical evaluation of an additional cylindrical Bessel function, the recurrence relation,
\begin{equation}
J_{\nu+1}(z) = \dfrac{2\nu}{z}J_\nu(z) - J_{\nu-1}(z),
\end{equation}

\noindent may be employed to reduce the number of evaluations down to two. Normally, the forward recurrence of cylindrical Bessel functions is ill-advised; however, for our purposes only two extra terms are used and their overall contribution is attenuated by the denominator, $(\ell+\tfrac{1}{2})^\nu$. Since $\frac{\pi}{2} < 2(\sqrt{2}-1)\pi < \pi$, for numerical evaluation on $[\frac{\pi}{2},\pi]$ we invoke the symmetry relation $P_\ell(\cos(\pi-\theta)) = (-1)^\ell P_\ell(\cos\theta)$.

\paragraph{Stable evaluation for $\theta\approx0$}
Due to the singular nature of the kernel, the numerical evaluation $P_\ell(\cos\theta) - 1$ for $\theta\approx0$ to high relative accuracy is essential. This can be achieved by a local expansion at $\theta = 0$ through the series representation
\begin{equation}
P_\ell(\cos\theta) = \sum_{k=0}^\ell (-1)^k\binom{\ell}{k}\binom{\ell+k}{k}\sin^{2k}\tfrac{\theta}{2}.
\end{equation}

\noindent Numerical evaluation of this series allows for the recovery of high relative accuracy for small angles, 
which is crucial for the accurate evaluation of the weakly singular integral. 
Furthermore, we may stably divide $P_\ell(\cos\theta)-1$ by $\sin^2\frac{\theta}{2}$ and the ratio tends to a constant as $\theta\to0^+$.

\subsection{Numerical experiments}

\begin{figure}[t]
\begin{center}
\begin{tabular}{cc}
\hspace*{-0.2cm}\includegraphics[width=0.53\textwidth]{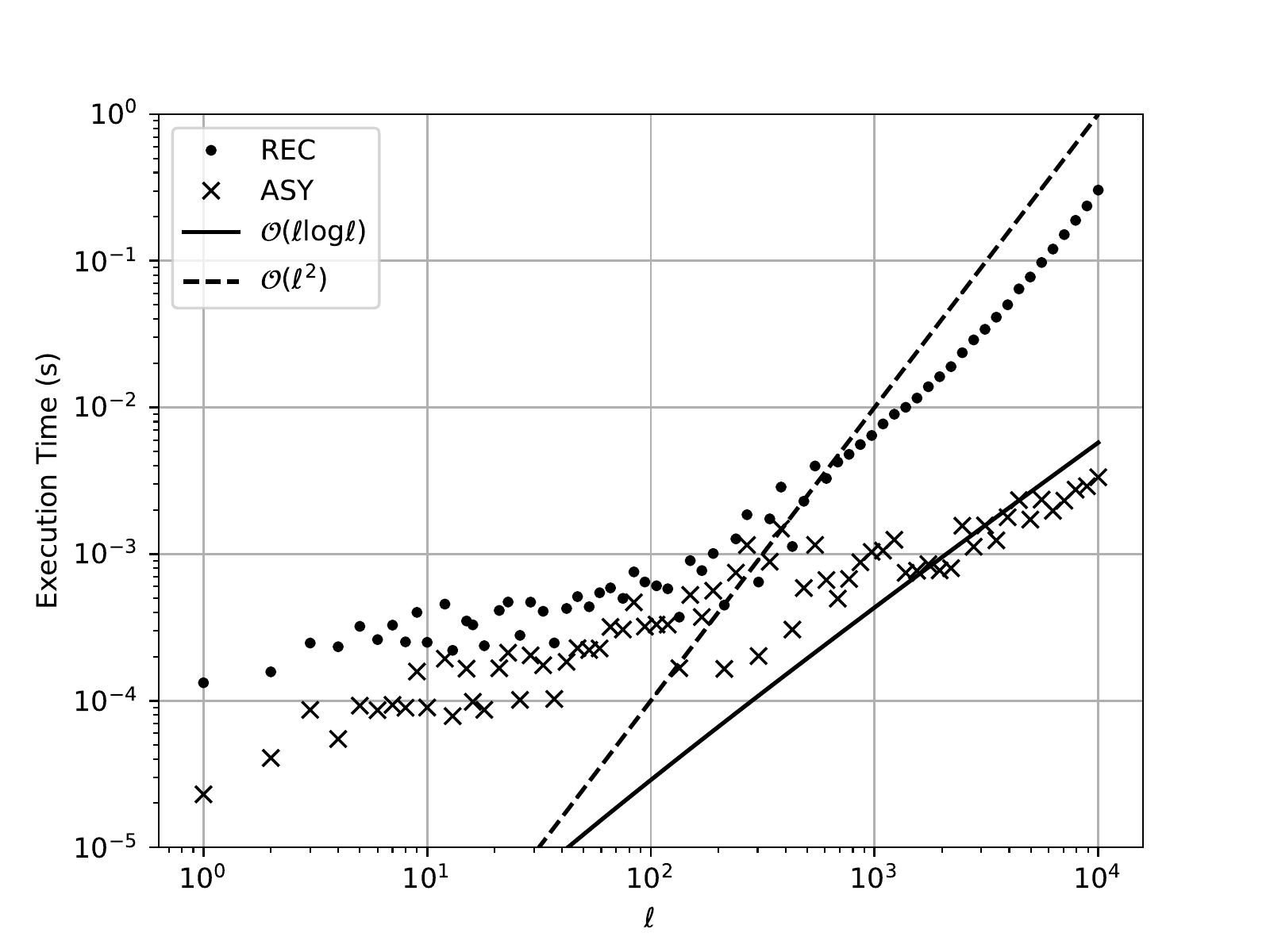}&
\hspace*{-0.65cm}\includegraphics[width=0.53\textwidth]{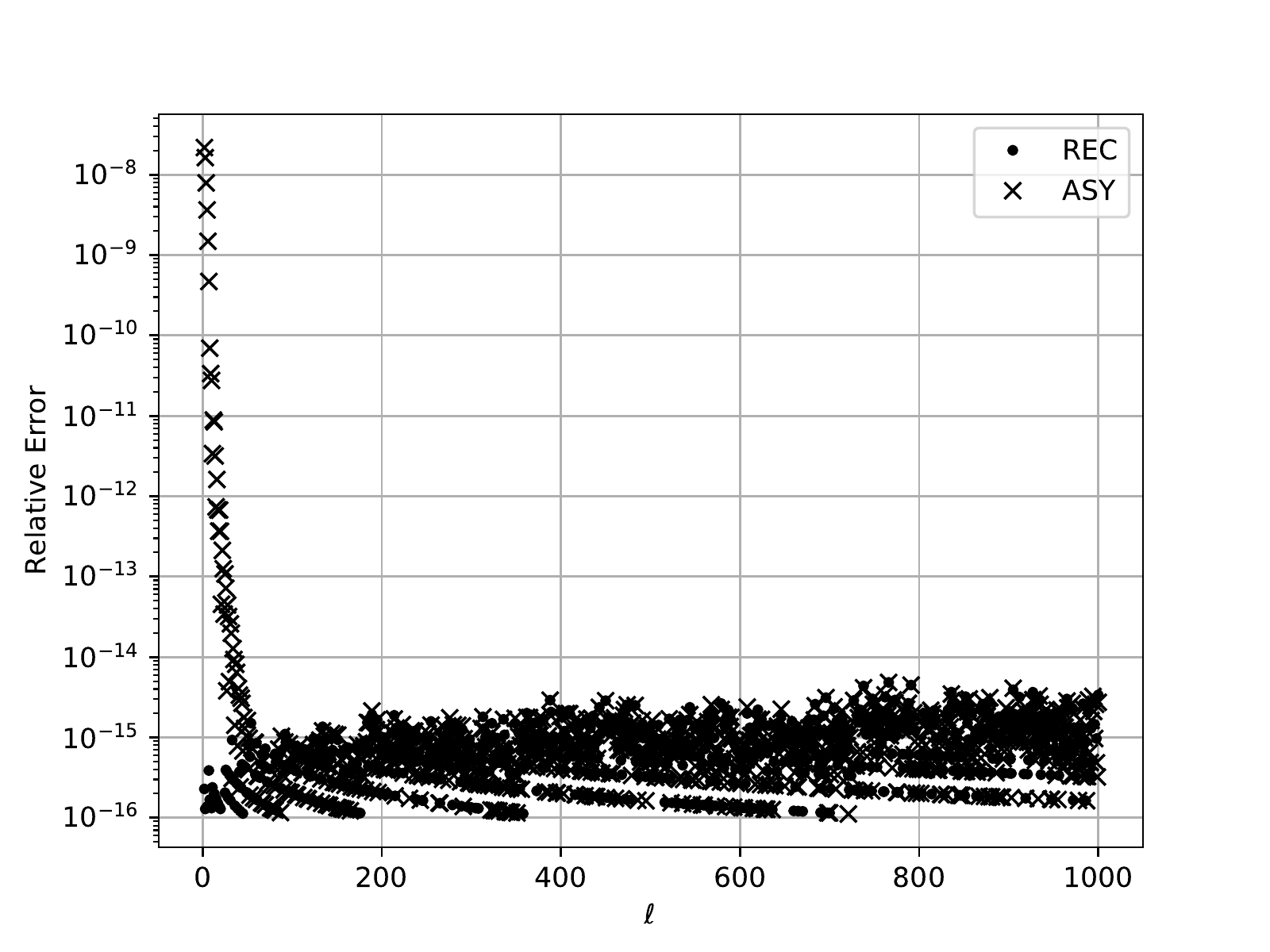}\\
\end{tabular}
\caption{Left: calculation times for the numerical evaluation of $\lambda_1(\ell)$ using the recurrence relation (REC) for the Legendre polynomials and the asymptotics (ASY). 
Right: approximation to the relative error in the numerical evaluation of $\lambda_1(\ell)$ in IEEE double precision floating-point arithmetic by comparison 
with numerical results in $256$-bit extended precision floating-point arithmetic. In both plots, $\alpha = -0.5$.}
\label{fig:spectrum_time_error}
\end{center}
\end{figure}

\begin{figure}[t]
\begin{center}
\begin{tabular}{cc}
\hspace*{-0.2cm}\includegraphics[width=0.53\textwidth]{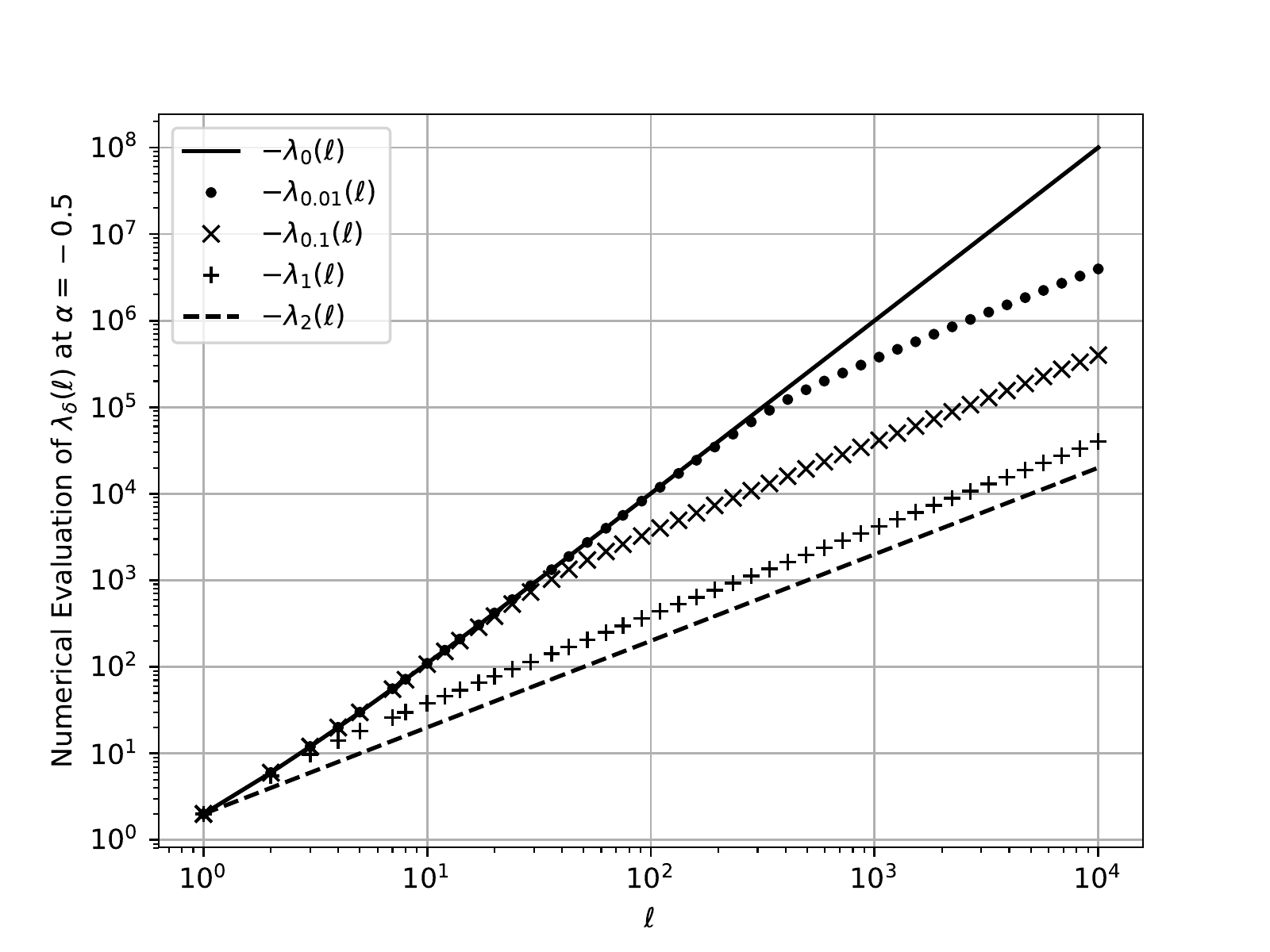}&
\hspace*{-0.65cm}\includegraphics[width=0.53\textwidth]{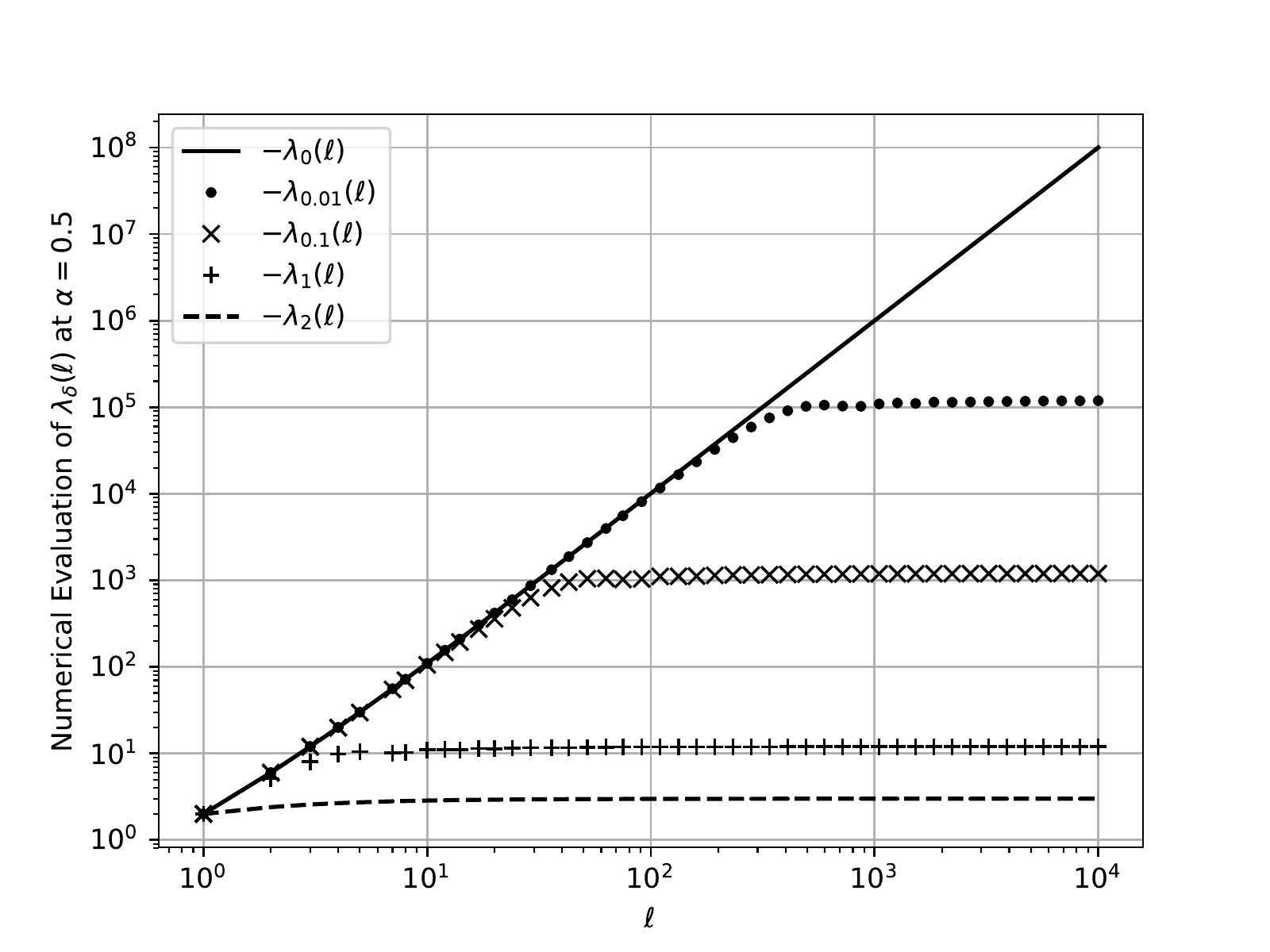}\\
\end{tabular}
\caption{Spectral data $\lambda_\delta(\ell)$ for the nonlocal operator $\LL_\delta$ with different values of the strength of the singularity 
$\alpha$ and the horizon $\delta$. Left: for $\alpha = -0.5$. Right: for $\alpha = 0.5$.}
\label{fig:plot_spectrum}
\end{center}
\end{figure}

We begin by illustrating calculation times and relative accuracy in the computation of the spectrum \eqref{eigenvalues2} by quadrature using the recurrence 
and asymptotic methods for $\alpha=-0.5$ and $\delta=1$.
For the accuracy, we compute the ``exact eigenvalues'' $\lambda_1(\ell)$ for $1\leq\ell\leq10^3$ using the recurrence in $256$-bit 
extended precision floating-point arithmetic.
We show in Figure \ref{fig:spectrum_time_error} the calculation times (left) and the relative accuracy (right).
As expected the recurrence has quadratic cost while the asymptotics have log-linear cost.
In terms of accuracy, the recurrence gives very good accuracy for all $\ell$ while the asymptotics are accurate only for $\ell>50$.

In the second experiment we compute the eigenvalues for different values of the strength of the singularity $\alpha$ and the horizon $\delta$.
We use the recurrence for $\ell\leq50$ and the asymptotics for $\ell>50$.
In Figure \ref{fig:plot_spectrum}, we plot (minus) the eigenvalues for $\alpha=\pm0.5$ and $\delta=0,\,0.01,\,0.1,\,1,$ and $2$.
This experiment demonstrates that nonlocal diffusion is asymptotically weaker than its local analogue: the larger $\alpha$ and $\delta$, the weaker.
This is reasonable because, for example, the average value of a function over a hemisphere has more inertia than the average value over an infinitesimally small region.
Note that for $\alpha=\pm0.5$ and $\delta=0$ we recover the eigenvalues $-\ell(\ell+1)$ of the local operator, while
for $\alpha=-0.5$ and $\delta=2$ we recover the eigenvalues $-2\ell$ of the nonlocal operator with integration over the entire sphere~\cite[Eq.~(3.74)]{atkinson2012}.

\section{Solving the nonlocal Poisson equation}

Before solving nonlocal time-dependent models in Section 5, we show how to solve the nonlocal Poisson equation 
with a mean condition for uniqueness,
\begin{equation}
\begin{array}{l}
\LL_\delta u(\theta,\varphi) = f(\theta,\varphi), \quad (\theta,\varphi)\in[0,\pi]\times[0,2\pi), \\\\
\dsp \int_{\Sph^2} u(\theta,\varphi)\ud\Omega = \dsp \int_{\Sph^2} f(\theta,\varphi)\ud\Omega.
\end{array}
\label{Poisson}
\end{equation}

\noindent We discretize colatitude and longitude with a uniform grid with $n+1$ points in the colatitudinal direction and $2n+1$ points in the longitudinal direction, 
and seek a solution $u(\theta,\varphi)$ of degree $n$ of the form
\begin{equation}
u(\theta,\varphi) = \sum_{\ell=0}^n \sum_{m=-\ell}^{+\ell} u_\ell^m Y_\ell^m(\theta,\varphi).
\end{equation}

\noindent The spherical harmonic coefficients $u_\ell^m$ populate a doubly triangular matrix but for computational purposes, we organize them into the array
\begin{equation}
U = \begin{pmatrix}
u_0^0 & u_1^{-1} & u_1^1 & u_2^{-2} & u_2^2 & \ldots & u_n^{-n} & u_n^n \\[3pt]
u_1^0 & u_2^{-1} & u_2^1 & u_3^{-2} & u_3^{2} & \ldots & 0 & 0 \\[3pt]
\vdots & \vdots & \vdots & \vdots & \vdots & \ddots & \vdots & \vdots \\[3pt]
u_{n-2}^0 & u_{n-1}^{-1} & u_{n-1}^1 & u_n^{-2} & u_n^{2} & \iddots & 0 & 0 \\[3pt]
u_{n-1}^0 & u_n^{-1} & u_n^1 & 0 & 0 & \ldots & 0 & 0 \\[3pt]
u_n^0 & 0 & 0 & 0 & 0 & \ldots & 0 & 0
\end{pmatrix}
\in\R^{(n+1)\times(2n+1)}.
\label{coefficients}
\end{equation}

\noindent The right-hand side $f$ is also expanded in spherical harmonics with coefficients $f_\ell^m$, which are stored in an array $F$.

Note that the mean of $u$ is given by
\begin{equation}
\int_{\Sph^2} u(\theta,\varphi)\ud\Omega
= u_0^0 Y_0^0
= \frac{u_0^0}{\sqrt{4\pi}}.
\label{mean}
\end{equation}

\noindent Therefore $u$ and $f$ have the same mean if and only if $u^0_0=f^0_0$.

\subsection{Nonlocal Laplace--Beltrami matrix}

\begin{figure}[t]
\centering
\includegraphics[scale=0.4]{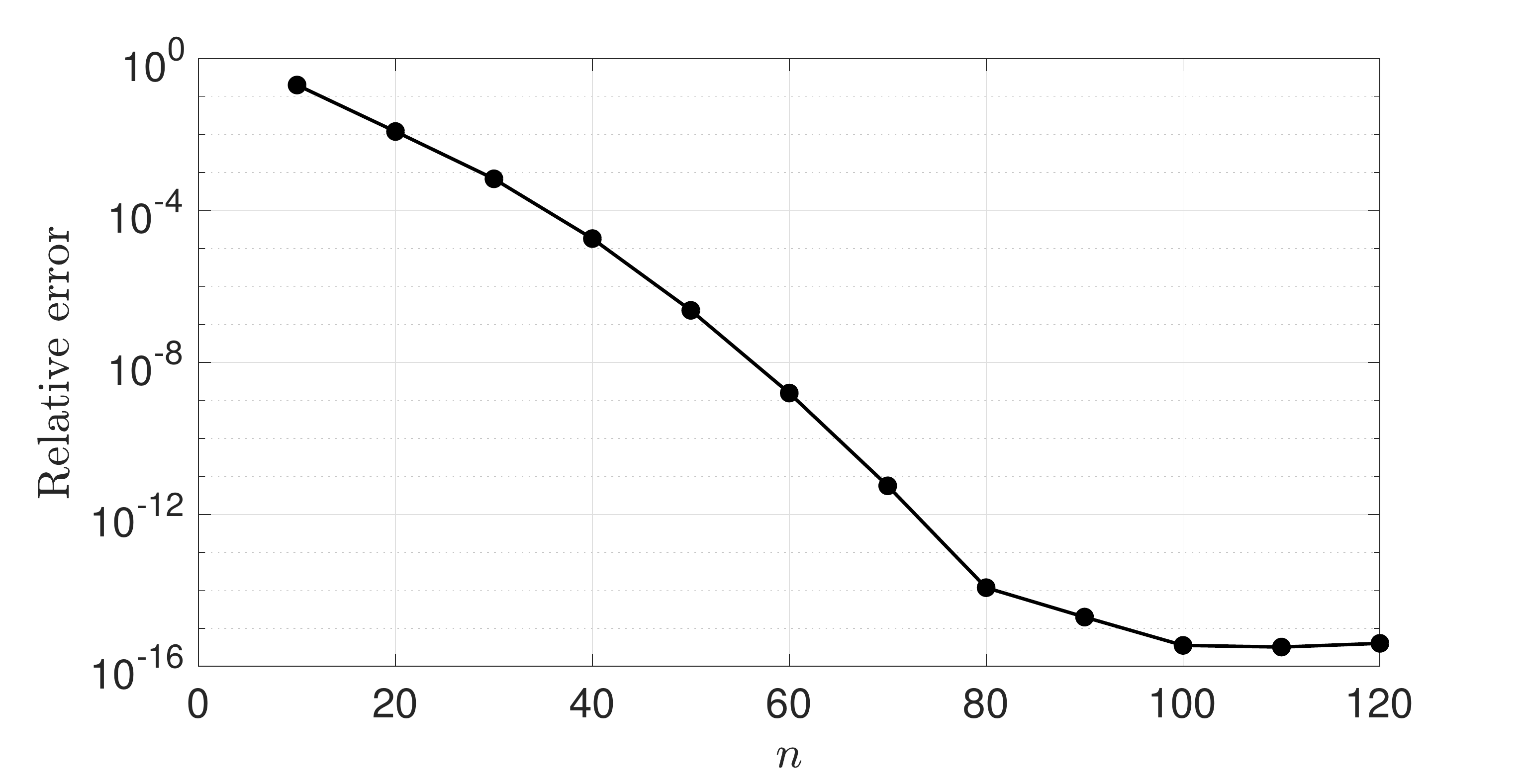}
\caption{Geometric decay of the $2$-norm relative error when solving the nonlocal Poisson equation $\eqref{Poisson}$
with right-hand side $\eqref{rhs}$, and parameters $\alpha=-0.5$ and $\delta=1.5$. 
One achieves machine accuracy for degrees $n\geq80$.}
\label{fig:nonlocal_poisson_cv}
\end{figure}

The Laplace--Beltrami operator $\LL_\delta$ acts diagonally on spherical harmonics so its discretization 
$L_\delta$ is a diagonal singular matrix ($\lambda_\delta(0)=0$), which we store as an $(n+1)\times(2n+1)$ matrix acting pointwise on the coefficients $U$, i.e.,
\begin{small}
\begin{equation}
L_\delta = \begin{pmatrix}
\lambda_\delta(0) & \lambda_\delta(1) & \lambda_\delta(1) & \lambda_\delta(2) & \lambda_\delta(2) & \ldots 
&\lambda_\delta(n) & \lambda_\delta(n) \\[3pt]
\lambda_\delta(1) & \lambda_\delta(2) & \lambda_\delta(2) &\lambda_\delta(3) & \lambda_\delta(3) & \ldots & 0 & 0 \\[3pt]
\vdots & \vdots & \vdots & \vdots & \vdots & \ddots & \vdots & \vdots \\[3pt]
\lambda_\delta(n-2) & \lambda_\delta(n-1)& \lambda_\delta(n-1) &\lambda_\delta(n) & \lambda_\delta(n) & \iddots & 0 & 0 \\[3pt]
\lambda_\delta(n-1) & \lambda_\delta(n) & \lambda_\delta(n) & 0 & 0 & \ldots & 0 & 0 \\[3pt]
\lambda_\delta(n) & 0 & 0 & 0 & 0 & \ldots & 0 & 0
\end{pmatrix}
\in\R^{(n+1)\times(2n+1)}.
\label{matrix}
\end{equation}
\end{small}

\noindent To make it nonsingular, we impose the mean condition by replacing $\lambda_\delta(0)$ by $1$, 
which gives $u_0^0 = f_0^0$.
We then solve 
\begin{equation}
L_\delta U = F 
\label{discretePoisson}
\end{equation}

\noindent by simply inverting all the nonzero entries of $L_\delta$ pointwise, that is, $U = L_\delta^{-1} F$ with 
\begin{small}
\begin{equation}
L_\delta^{-1} = \begin{pmatrix}
1 & \lambda_\delta^{-1}(1) & \lambda_\delta^{-1}(1) & \lambda_\delta^{-1}(2) & \lambda_\delta^{-1}(2) & \ldots 
&\lambda_\delta^{-1}(n) & \lambda_\delta^{-1}(n) \\[3pt]
\lambda_\delta^{-1}(1) & \lambda_\delta^{-1}(2) & \lambda_\delta^{-1}(2) &\lambda_\delta^{-1}(3) & \lambda_\delta^{-1}(3) & \ldots & 0 & 0 \\[3pt]
\vdots & \vdots & \vdots & \vdots & \vdots & \ddots & \vdots & \vdots \\[3pt]
\lambda_\delta^{-1}(n-2) & \lambda_\delta^{-1}(n-1)& \lambda_\delta^{-1}(n-1) &\lambda_\delta^{-1}(n) & \lambda_\delta^{-1}(n) & \iddots & 0 & 0 \\[3pt]
\lambda_\delta^{-1}(n-1) & \lambda_\delta^{-1}(n) & \lambda_\delta^{-1}(n) & 0 & 0 & \ldots & 0 & 0 \\[3pt]
\lambda_\delta^{-1}(n) & 0 & 0 & 0 & 0 & \ldots & 0 & 0
\end{pmatrix}.
\end{equation}
\end{small}

Note that our method could also be applied to the nonlocal Helmholtz equation $\LL_\delta u + c^2 u = 0$.
Therefore, implicit-explicit time-stepping schemes could also be used to solve time-dependent equations in Section 5.

\subsection{Numerical experiments}

\begin{figure}[t]
\begin{center}
\begin{tabular}{ccc}
Right-hand side\hspace*{1cm} & Nonlocal solution\hspace*{1cm} & Local solution\\
\includegraphics[width=0.24\textwidth]{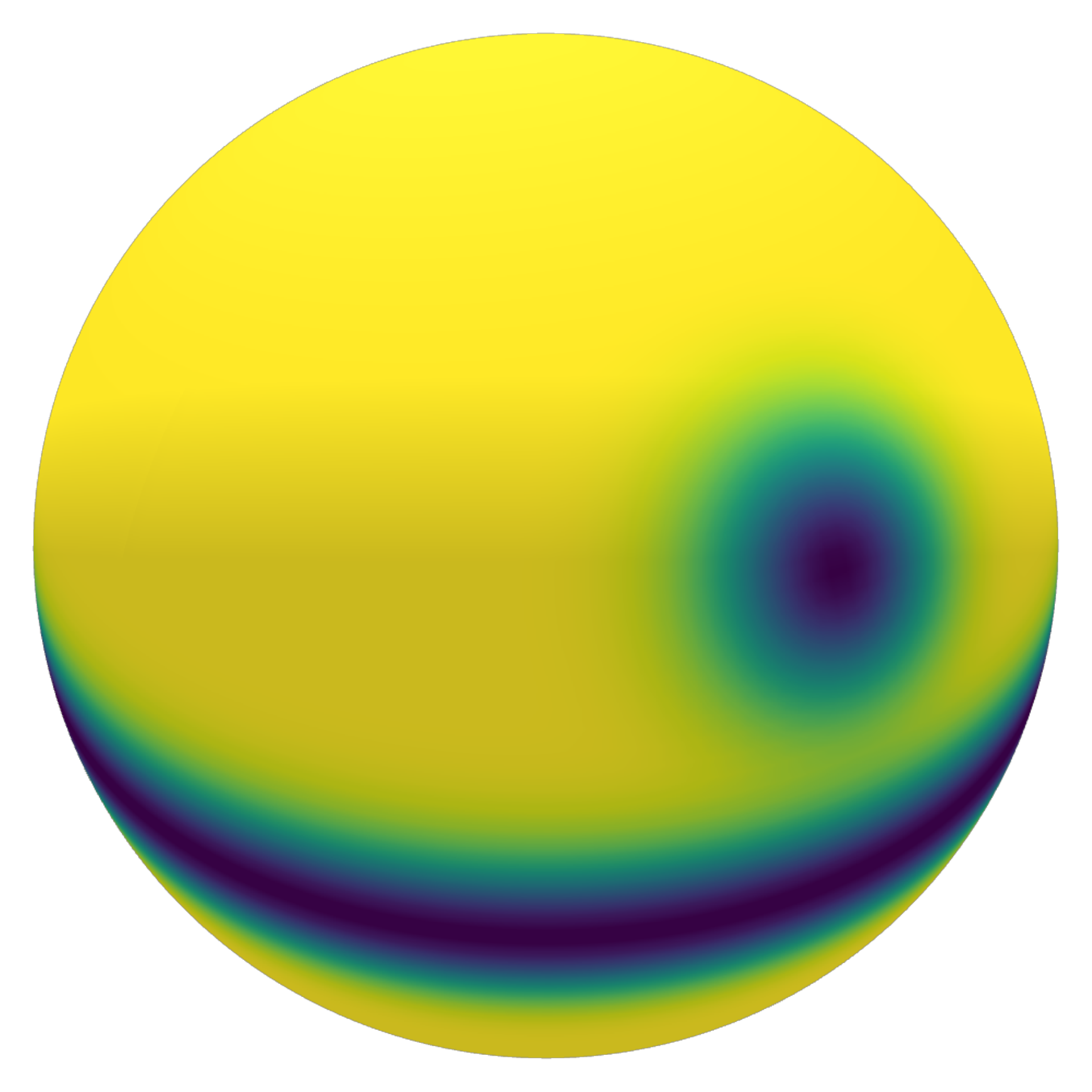}\hspace*{1cm}&
\includegraphics[width=0.24\textwidth]{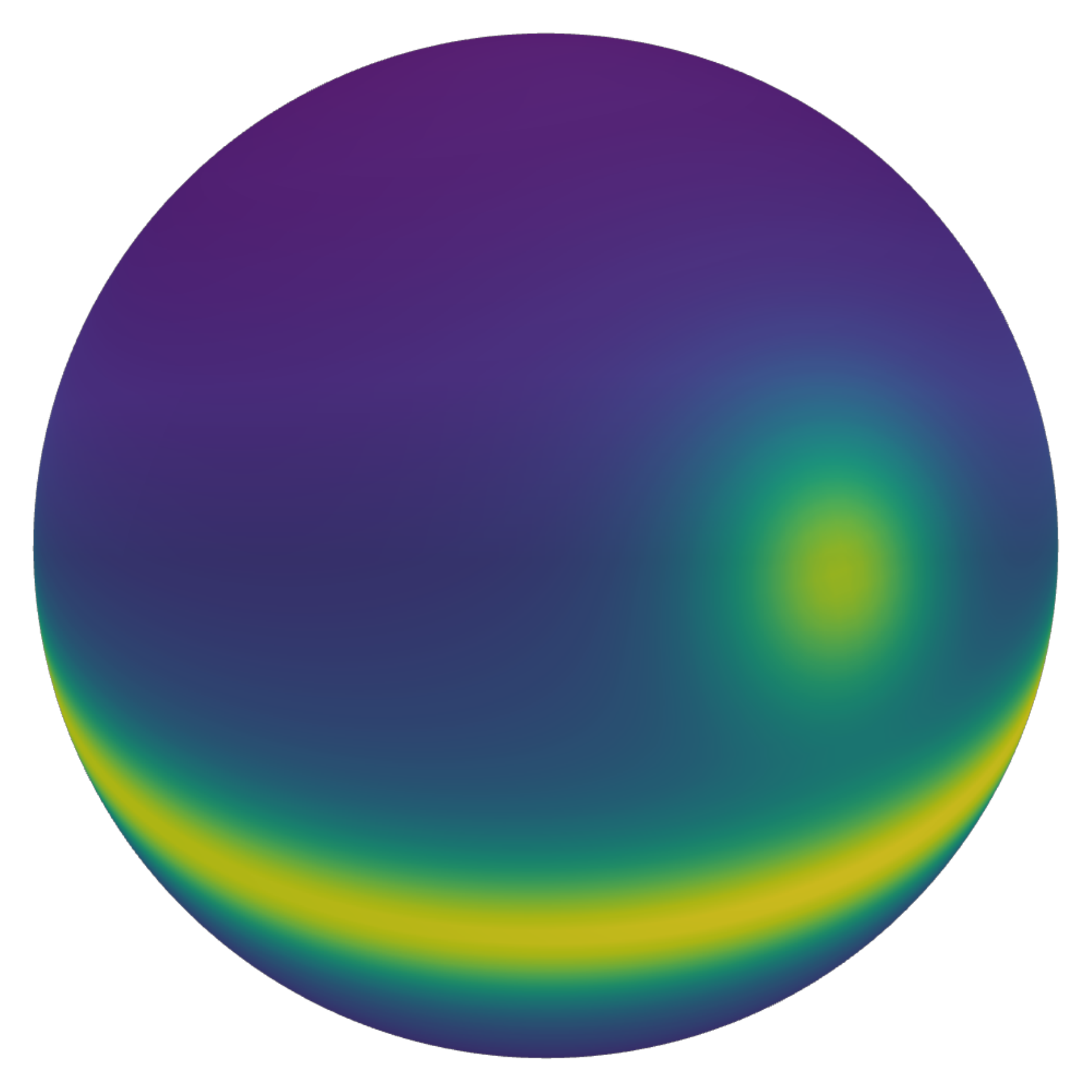}\hspace*{1cm}&
\includegraphics[width=0.24\textwidth]{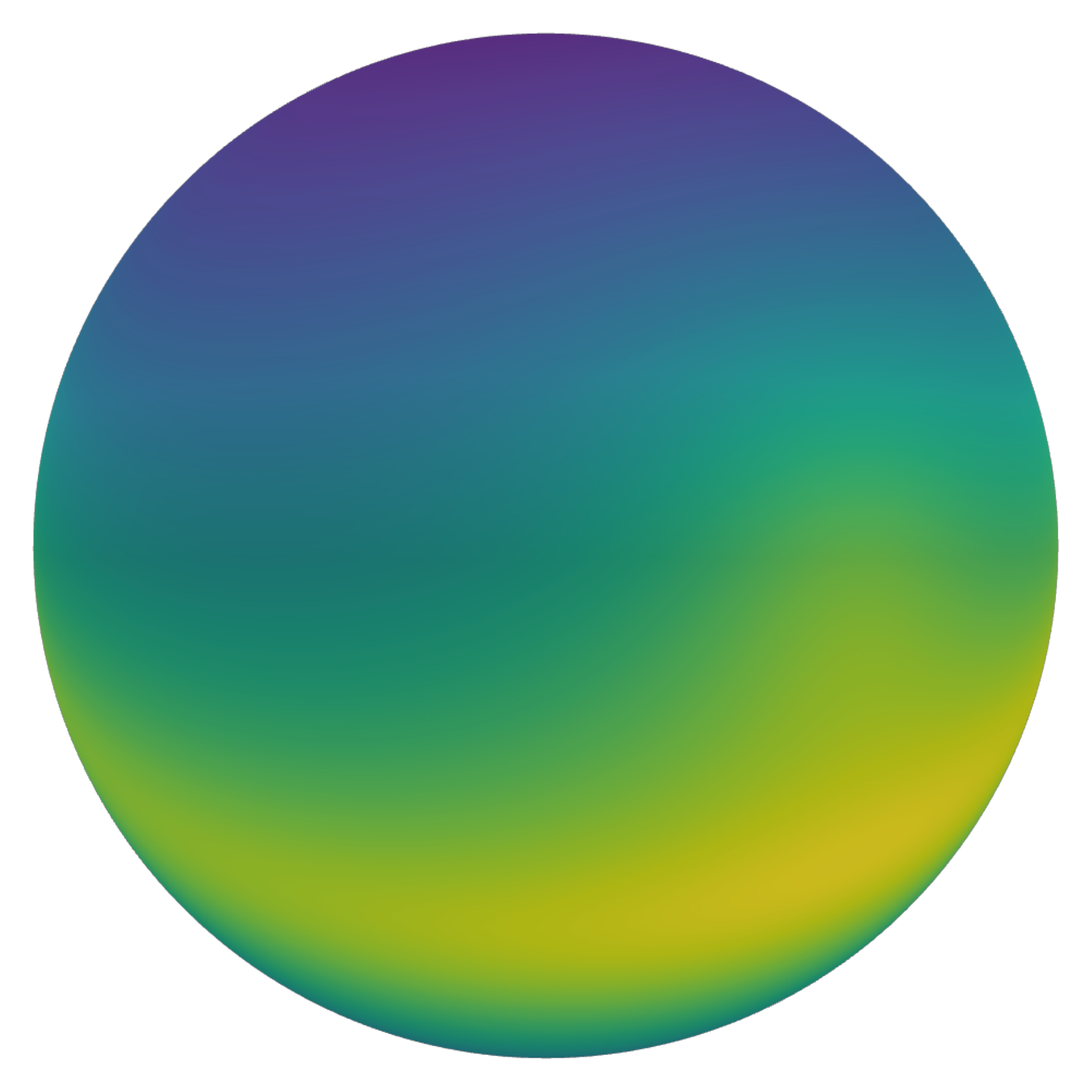}\\
\end{tabular}
\caption{``Death star'' right-hand side with nonlocal ($\alpha=0$, $\delta=1.5$) and local solutions, 
which clearly shows that nonlocal diffusion leads to sharper interfaces. The color is scaled to the extrema of the data.}
\label{fig:nonlocal_poisson_solution}
\end{center}
\end{figure}

We solve the nonlocal Poisson equation \eqref{Poisson} with a ``death star'' right-hand side
\begin{equation}
f(x, y, z) = -e^{-30((x-1/4)^2 + (y-\sqrt{11}/4)^2 + (z-1/4)^2)} - e^{-50z^2},
\label{rhs}
\end{equation}

\noindent and parameters $\alpha=0$ and $\delta=1.5$.
We compute the ``exact solution'' by numerically solving \eqref{discretePoisson} on a very fine grid.
We then compute numerical solutions for $n=10,20,30,\ldots,120$ and measure the $2$-norm relative error (measured on the coefficients) 
between numerical and exact solutions. 
We expect spectral convergence and this is what we observe in Figure \ref{fig:nonlocal_poisson_cv}.
The right-hand side together with the local and nonlocal solutions are shown in Figure \ref{fig:nonlocal_poisson_solution}.

\section{Solving nonlocal time-dependent equations}

For nonlinear nonlocal time-dependent equations of the form \eqref{PDE}, we seek solutions of the form
\begin{equation}
u(t, \theta,\varphi) = \sum_{\ell=0}^n \sum_{m=-\ell}^{+\ell} u_\ell^m(t) Y_\ell^m(\theta,\varphi).
\end{equation}

\noindent We obtain a coupled system of nonlinear {\em ordinary} differential equations,
\begin{equation}
U_t = \epsilon^2 L_\delta U + N(U), \quad U(0) = U_0,
\label{ODEs}
\end{equation}

\noindent where $U(t)$ represents the $(n+1)\times(2n+1)$ matrix of spherical harmonic coefficients $u_\ell^m(t)$, 
and the nonlinearity $N(U)$ is evaluated on the grid using a fast spherical harmonic transform.

\subsection{Fast spherical harmonic transforms}

\begin{figure}[t]
\begin{center}
\begin{tikzpicture}[scale=0.657]
\draw[black, thick]
    (0,0) -- (6,0)
    (1,1) -- (6,1)
    (2,2) -- (6,2)
    (3,3) -- (6,3)
    (5,5) -- (6,5)
    ;
\filldraw[black]
    (0,0) circle (2pt)
    (6,0) circle (2pt)
    (1,1) circle (2pt)
    (6,1) circle (2pt)
    (2,2) circle (2pt)
    (6,2) circle (2pt)
    (3,3) circle (2pt)
    (6,3) circle (2pt)
    (5,5) circle (2pt)
    (6,5) circle (2pt)
    (6,6) circle (2pt)
    ;
\node (ellipsis1) at (4.5,4) {$\iddots$};
\node (ellipsis2) at (5.5,4) {$\vdots$};
\node[anchor=west] (zero) at (6.25,0) {$\tilde{P}_\ell^0$};
\node[anchor=west] (one) at (6.25,1) {$\tilde{P}_\ell^1$};
\node[anchor=west] (two) at (6.25,2) {$\tilde{P}_\ell^2$};
\node[anchor=west] (three) at (6.25,3) {$\tilde{P}_\ell^3$};
\node[anchor=west] (penultimate) at (6.25,5) {$\tilde{P}_\ell^{\ell-1}$};
\node[anchor=west] (ultimate) at (6.25,6) {$\tilde{P}_\ell^\ell$};
\draw[black, thick]
    (9,0) -- (15,0)
    (10,1) -- (15,1)
    (9,2) -- (15,2)
    (10,3) -- (15,3)
    (9,5) -- (15,5)
    (10,6) -- (15,6)
    ;
\filldraw[black]
    (9,0) circle (2pt)
    (15,0) circle (2pt)
    (10,1) circle (2pt)
    (15,1) circle (2pt)
    (9,2) circle (2pt)
    (15,2) circle (2pt)
    (10,3) circle (2pt)
    (15,3) circle (2pt)
    (9,5) circle (2pt)
    (15,5) circle (2pt)
    (10,6) circle (2pt)
    (15,6) circle (2pt)
    ;
\node (ellipsis1) at (11,4) {$\vdots$};
\node (ellipsis2) at (14,4) {$\vdots$};
\node[anchor=west] (zero) at (15.25,0) {$\tilde{P}_\ell^0$};
\node[anchor=west] (one) at (15.25,1) {$\tilde{P}_\ell^1$};
\node[anchor=west] (two) at (15.25,2) {$\tilde{P}_\ell^0$};
\node[anchor=west] (three) at (15.25,3) {$\tilde{P}_\ell^1$};
\node[anchor=west] (penultimate) at (15.25,5) {$\tilde{P}_\ell^0$};
\node[anchor=west] (ultimate) at (15.25,6) {$\tilde{P}_\ell^1$};
\draw[black, thick]
    (18,0) -- (24,0)
    (18,1) -- (23,1)
    (18,2) -- (24,2)
    (18,3) -- (23,3)
    (18,5) -- (24,5)
    (18,6) -- (23,6)
    ;
\filldraw[black]
    (18,0) circle (2pt)
    (24,0) circle (2pt)
    (18,1) circle (2pt)
    (23,1) circle (2pt)
    (18,2) circle (2pt)
    (24,2) circle (2pt)
    (18,3) circle (2pt)
    (23,3) circle (2pt)
    (18,5) circle (2pt)
    (24,5) circle (2pt)
    (18,6) circle (2pt)
    (23,6) circle (2pt)
    ;
\node (ellipsis1) at (19,4) {$\vdots$};
\node (ellipsis2) at (22,4) {$\vdots$};
\node[anchor=west] (zero) at (24.25,0) {$T_\ell$};
\node[anchor=west] (one) at (23.25,1) {$\sin\theta U_\ell$};
\node[anchor=west] (two) at (24.25,2) {$T_\ell$};
\node[anchor=west] (three) at (23.25,3) {$\sin\theta U_\ell$};
\node[anchor=west] (penultimate) at (24.25,5) {$T_\ell$};
\node[anchor=west] (ultimate) at (23.25,6) {$\sin\theta U_\ell$};
\node (firstarrow) at (8.1,3) {$\Longrightarrow$};
\node (secondarrow) at (17.1,3) {$\Longrightarrow$};
\end{tikzpicture}
\caption{The spherical harmonic transform proceeds in two steps.
Firstly, normalized associated Legendre functions are converted to normalized associated Legendre functions of order zero and one. 
Then, these intermediate expressions are re-expanded in trigonometric form. 
We use the notation $T_\ell(\cos\theta) = \cos(\ell\theta)$ and $\sin\theta U_\ell(\cos\theta) = \sin((\ell+1)\theta)$.}
\label{fig:SHT}
\end{center}
\end{figure}
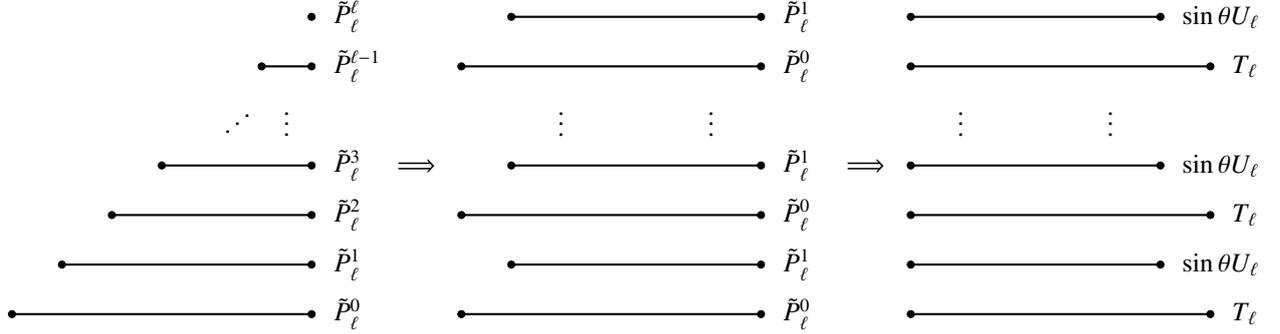

There are many algorithms available to accelerate synthesis and analysis on the sphere. 
Our particular choice is the one described by Slevinsky in~\cite{slevinsky2017a} due to the backward stability that is important for partial differential equations of evolution. 
We refer the interested reader to~\cite{slevinsky2017a} for a complete analysis and description of the implementation of the fast spherical harmonic transform.

The steps required by the spherical harmonic connection problem are illustrated in Figure~\ref{fig:SHT}. 
At first, the butterfly algorithm converts higher-order layers of the spherical harmonics into expansions with orders zero and one. 
Then, these coefficients are rapidly transformed into their Fourier coefficients by the Fast Multipole Method.
Total pre-computation requires at best $\mathcal{O}(n^3\log n)$ flops; and, the asymptotically optimal execution time of $\mathcal{O}(n^2\log^2 n)$ is rigorously proved 
via connection to Fourier integral operators, though the asymptotically optimal scaling is anticipated to set it for bandlimits beyond $n\ge\OO(20,000)$. 
Once a spherical harmonic expansion is converted to a bivariate Fourier series, FFTs are able to synthesize function samples at equispaced points-in-angle.

\subsection{High-order time-stepping with exponential integrators}

Time is discretized with a uniform time-step $h$ and the problem is to find the spherical harmonic coefficients $U^{k+1}$ of $u$ at $t_{k+1}=(k+1)h$
from the coefficients $U^{k}$ at $t_{k}=kh$.
Since the linear part $L_\delta$ in \eqref{ODEs} is diagonal, exponential integrators are particularly efficient as the computation of the matrix exponential is equivalent to pointwise exponentiation of the spectrum.
For diagonal problems, Montanelli and Bootland recently demonstrated~\cite{montanelli2017a} that one of the most
effective choices is the 
ETDRK4 scheme of Cox and Matthews~\cite{cox2002}. 
The formula for this scheme is:
\begin{equation}
\begin{array}{l}
A^k = e^{\frac{h}{2}L_\delta}U^k + L_\delta^{-1}(e^{hL_\delta} - I)N(U^k), \\[10pt]
B^k = e^{\frac{h}{2}L_\delta}U^k + L_\delta^{-1}(e^{hL_\delta} - I)N(A^k), \\[10pt]
C^k = e^{\frac{h}{2}L_\delta}A^k + L_\delta^{-1}(e^{hL_\delta} - I)\big[2N(B^k) - N(U^k)\big], \\[10pt]
U^{k+1} = e^{hL_\delta}U^k + f_1(hL_\delta)N(U^k) + 2f_2(hL_\delta)\big[N(A^k) + N(B^k)\big] + f_3(hL_\delta)N(C^k),
\end{array}
\label{ETDRK4}
\end{equation}

\noindent where the coefficients $f_1$, $f_2$ and $f_3$ are
\begin{equation}
\begin{array}{l}
f_1(hL_\delta) = h^{-2}L_\delta^{-3}[-4I - hL_\delta + e^{hL_\delta}(4 - 3hL_\delta + (hL_\delta)^2)], \\[10pt]
f_2(hL_\delta) = h^{-2}L_\delta^{-3}[2I + hL_\delta + e^{hL_\delta}(-2I + hL_\delta)], \\[10pt]
f_3(hL_\delta) = h^{-2}L_\delta^{-3}[-4I - 3hL_\delta - (hL_\delta)^2 + e^{hL_\delta}(4I - hL_\delta)].
\end{array}
\end{equation}

\noindent The stable evaluation of the coefficients may be performed using the contour integral method suggested by 
Kassam and Trefethen~\cite{kassam2005}. When working with an eigenfunction expansion, such as the spherical harmonic expansions, it is even more efficient to Taylor expand $f_1$, $f_2$ 
and $f_3$ for small argument to recover pointwise evaluation to high relative accuracy. Stability properties of the ETDRK4 scheme have been studied by Du and Zhu in \cite{du2005}. 
Note that due to our choice of spherical harmonics for the basis, there is no severe time-step restriction because there are no spurious 
eigenvalues from discretizing the nonlocal Laplace--Beltrami operator.\footnote{In \cite{montanelli2017b}, Montanelli and Nakatsukasa discretized the local 
Laplace--Beltrami operator using the DFS method. The eigenvalues of their Laplace--Beltrami matrix are all real and nonpositive;
some of them are spectrally accurate approximations to the eigenvalues $-\ell(\ell+1)$, but some others, the so-called \textit{outliers}, are of order $\OO(n^4)$.}

We note that for equations like \eqref{PDE}, one may include linear stabilizing terms to $L_\delta$ (that are simultaneously
diagonalizable with simple spectrum computation) to help improve the stability. For \eqref{PDE} with special structures, one 
may also use high-order energy-preserving ETDRK schemes for gradient flows and maximum-principle-preserving ETDRK schemes.

\subsection{Post-processing data with Ces\`aro summation}

Depending on the relative strength of the diffusive term $\epsilon^2\LL_\delta$ in, e.g., the nonlocal Allen--Cahn equation, the steady-state solution may be discontinuous. 
This is in contrast to, e.g., the local Allen--Cahn equation that tends to a steady-state consisting of the single constant $\pm1$ on the entire sphere.

As in the familiar Fourier case (see, e.g., \cite[Th.~9.3~of~Chap.~2]{zygmund1959}), the partial sums
\begin{equation}
S_n f(\theta,\varphi) = \sum_{\ell=0}^n \sum_{m=-\ell}^{+\ell} f_\ell^m Y_\ell^m(\theta,\varphi),
\end{equation}

\noindent show the Gibbs phenomenon at points of discontinuity, first described on the sphere by Weyl~\cite{weyl1910}.
The remedy in the Fourier case is to consider instead the arithmetic means of successive partial sums, which do not show the Gibbs phenomenon~\cite[Th.~3.4 of Chap.~3]{zygmund1959}.
To avoid the Gibbs phenomenon on the sphere, one has to consider Ces\`aro means of higher order (arithmetic means are Ces\`aro $(C,1)$ means).
In fact, Dai and Xu show that the $(C,\kappa)$ means of the spherical harmonic series defined by
\begin{equation}
S_n^\kappa f = \dfrac{1}{A_n^\kappa}\sum_{\ell=0}^n A_{n-\ell}^\kappa \sum_{m=-\ell}^{+\ell} f_\ell^m Y_\ell^m(\theta,\varphi),
\quad A_\ell^\kappa = \binom{\ell+\kappa}{\ell} = \dfrac{(\kappa+\ell)(\kappa+\ell-1)\cdots(\kappa+1)}{\ell!},
\end{equation}

\noindent completely remove the overshoot of the Gibbs phenomenon when $\kappa\geq2$ \cite[Th.~2.4.3]{dai2013}.
Therefore, we post-process our numerical solutions by rescaling the spherical harmonic coefficients using $(C, 2)$ means.\footnote{Note that Gelb proposed
a method for removing the Gibbs phenomenon for spherical harmonics in \cite{gelb1997}. However, her method is limited to the case where the position of the 
discontinuity is known.}

\subsection{Numerical experiments}

\paragraph{Allen--Cahn equation}

The Allen--Cahn equation, derived by Allen and Cahn in the 1970s, is a reaction-diffusion equation which
describes the process of phase separation in iron alloys~\cite{allen1979}. It was studied in the ball and on the sphere in~\cite{du2008}.
Our nonlocal version on the sphere is
\begin{equation}
u_t = \epsilon^2\LL_\delta u + u - u^3, 
\label{eq:NAC}
\end{equation}

\noindent with nonlocal diffusion $\epsilon^2\LL_\delta u$ and cubic reaction $u-u^3$.
The solution $u$ is the order parameter, a correlation function related to the positions of the different components of the alloy.
In our experiments, we take $\epsilon=0.1$, $\alpha = -0.5$, $\delta = 1.0$, and the initial condition
\begin{equation}
u(t=0,x,y,z) = \cos(10xy).
\label{eq:NACIC}
\end{equation}

\begin{figure}[t]
\begin{center}
\begin{tabular}{cc}
\hspace*{-0.2cm}\includegraphics[width=0.53\textwidth]{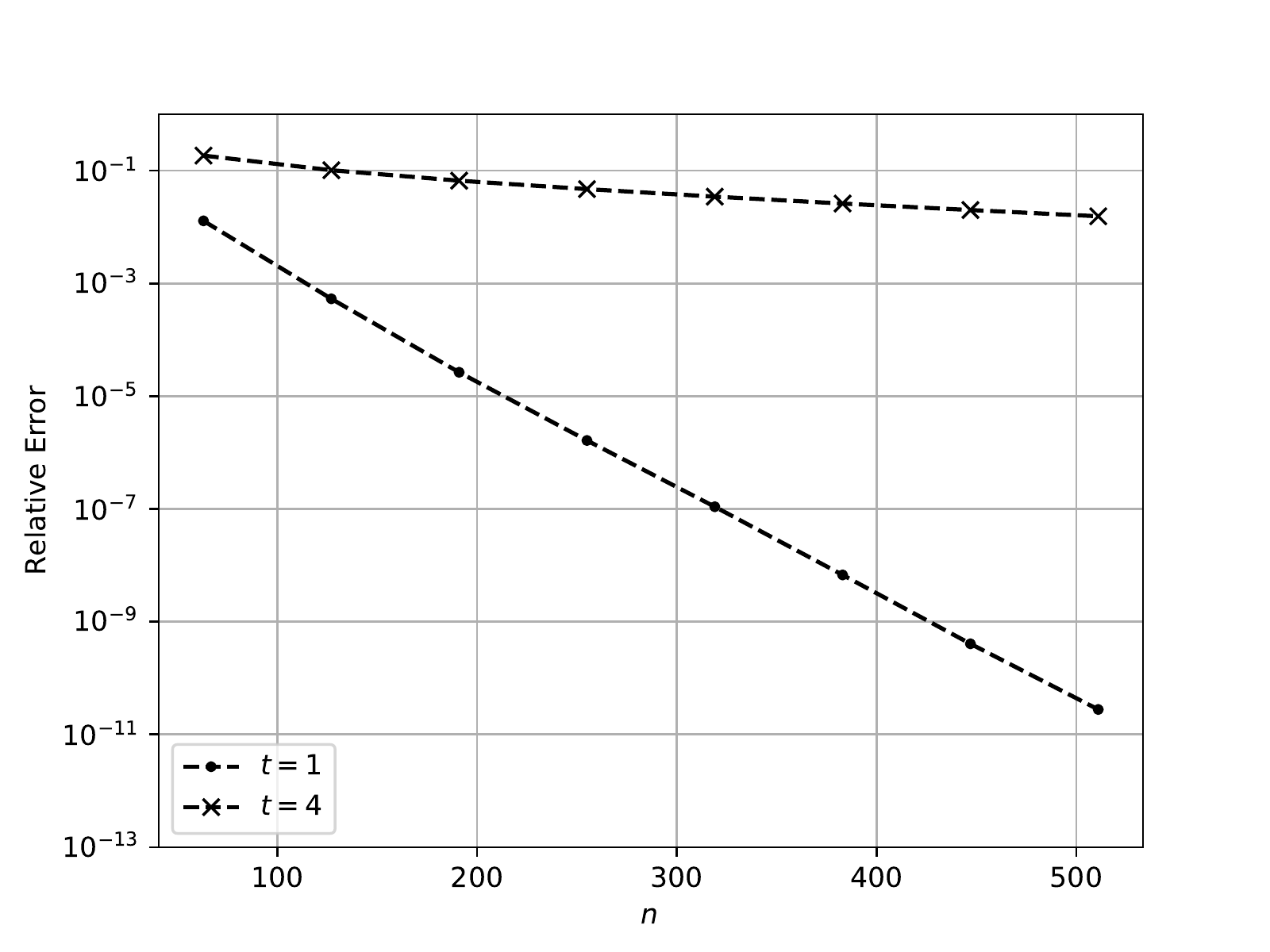}&
\hspace*{-0.65cm}\includegraphics[width=0.53\textwidth]{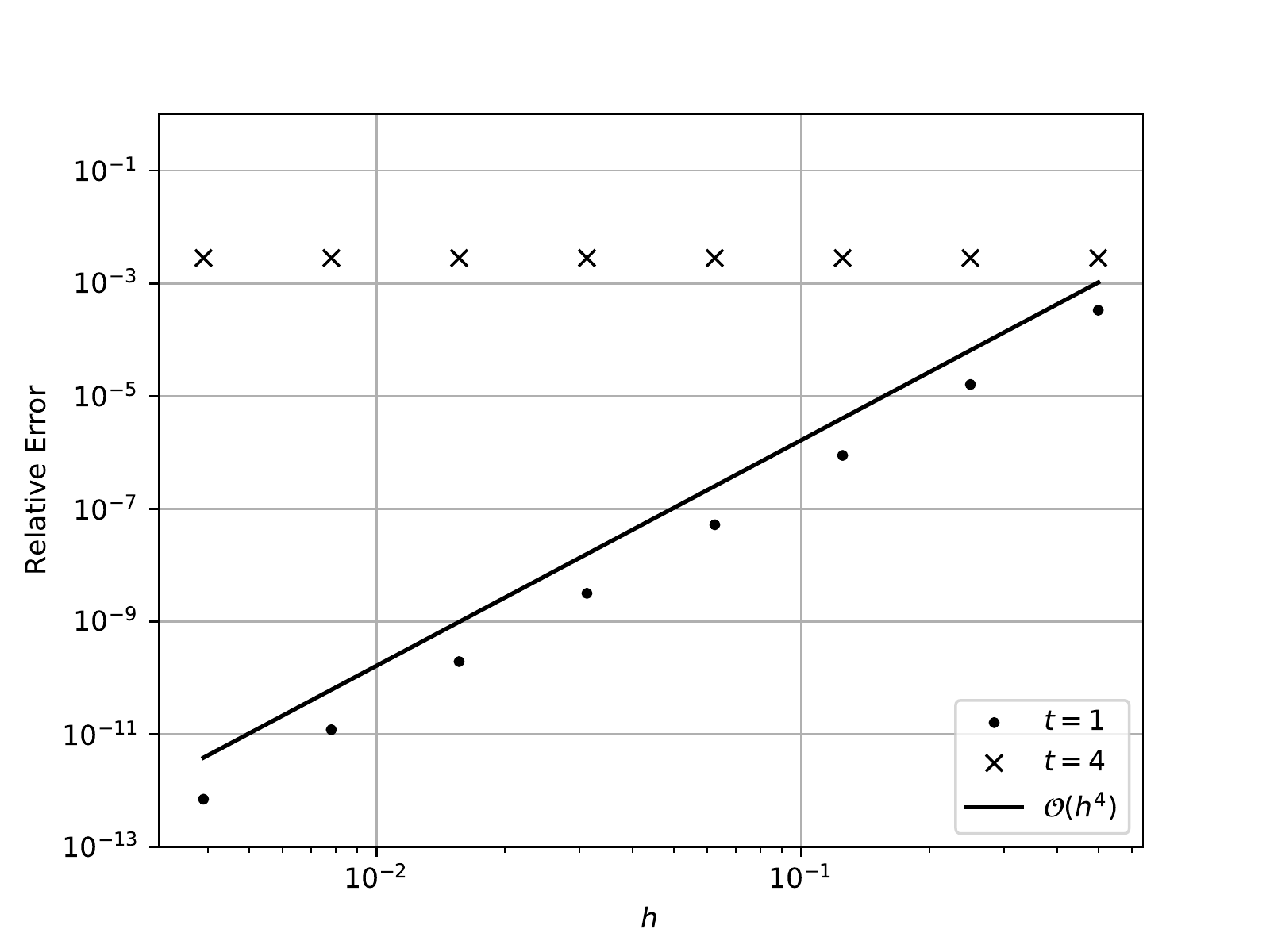}\\
\end{tabular}
\caption{Left: estimates of the spatial error demonstrating spectral convergence. Right: estimates of the temporal error demonstrating fourth-order convergence. In both plots, results show the $2$-norm relative error between the computed solution and a computed solution with a finer step size and a larger spherical harmonic degree.}
\label{fig:NAC_error}
\end{center}
\end{figure}

\begin{figure}[t]
\begin{center}
\begin{tabular}{ccccc}
$t=0$ & $t=1$ & $t=4$ & $t=16$ & $t=64$\\
\includegraphics[width=0.17\textwidth]{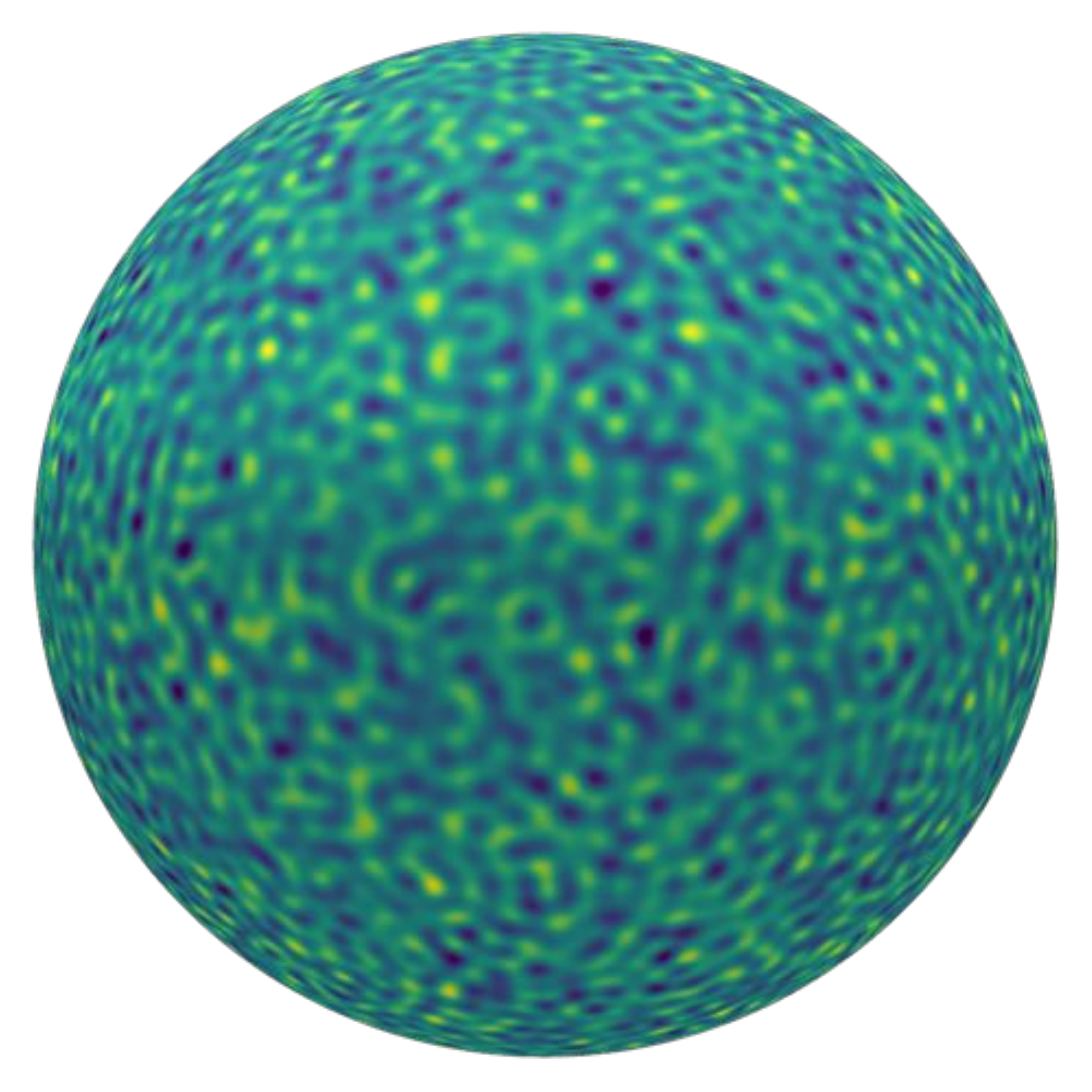}&
\includegraphics[width=0.17\textwidth]{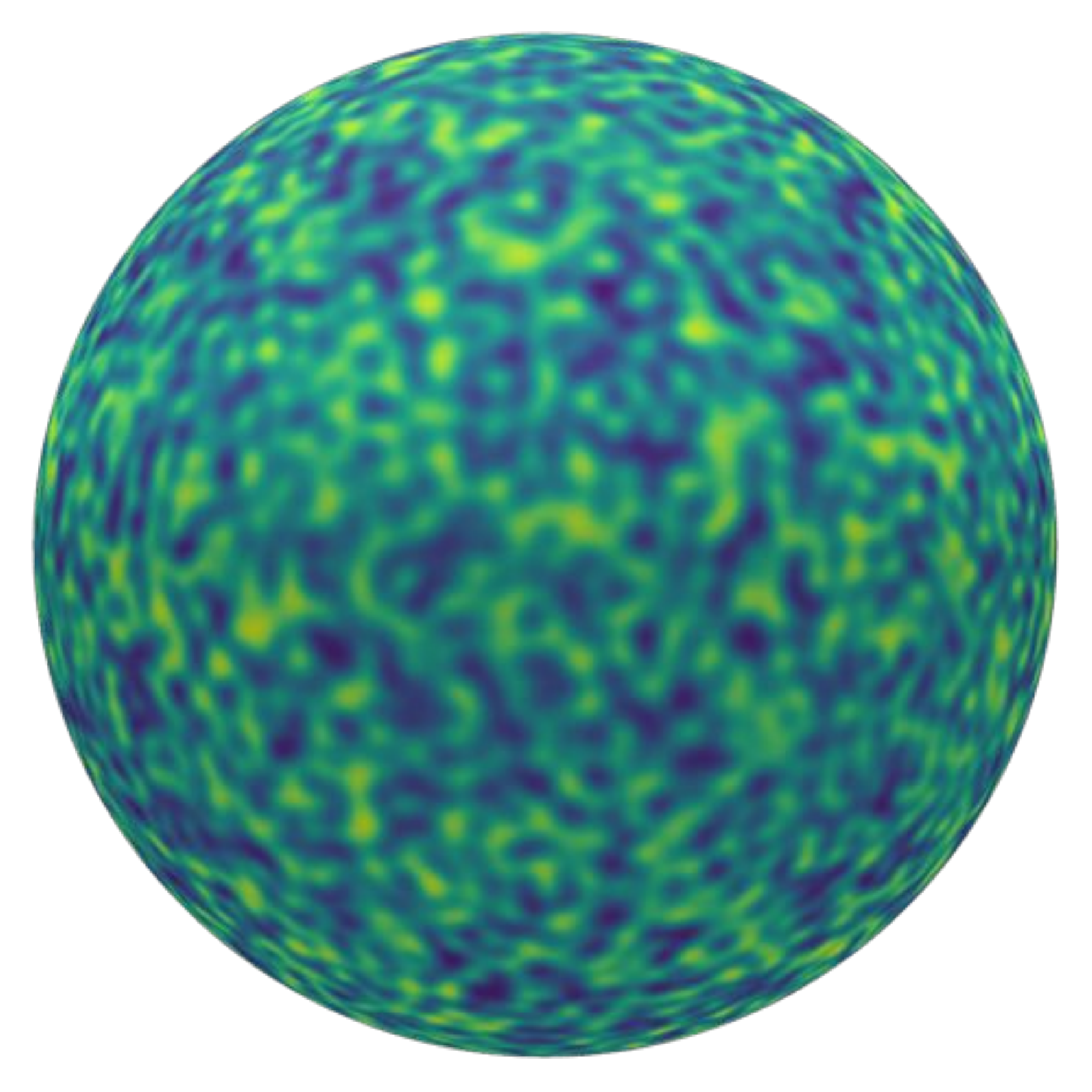}&
\includegraphics[width=0.17\textwidth]{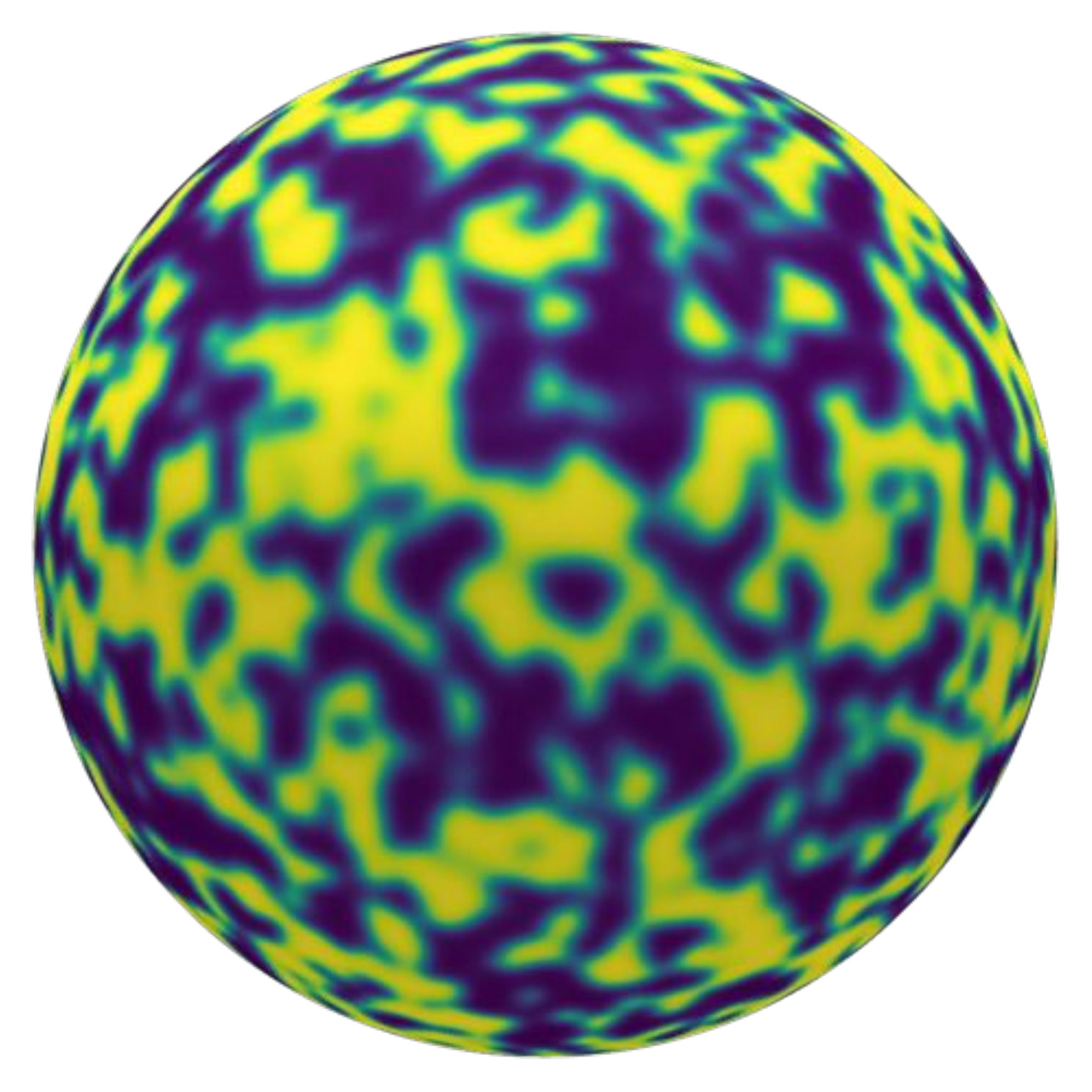}&
\includegraphics[width=0.17\textwidth]{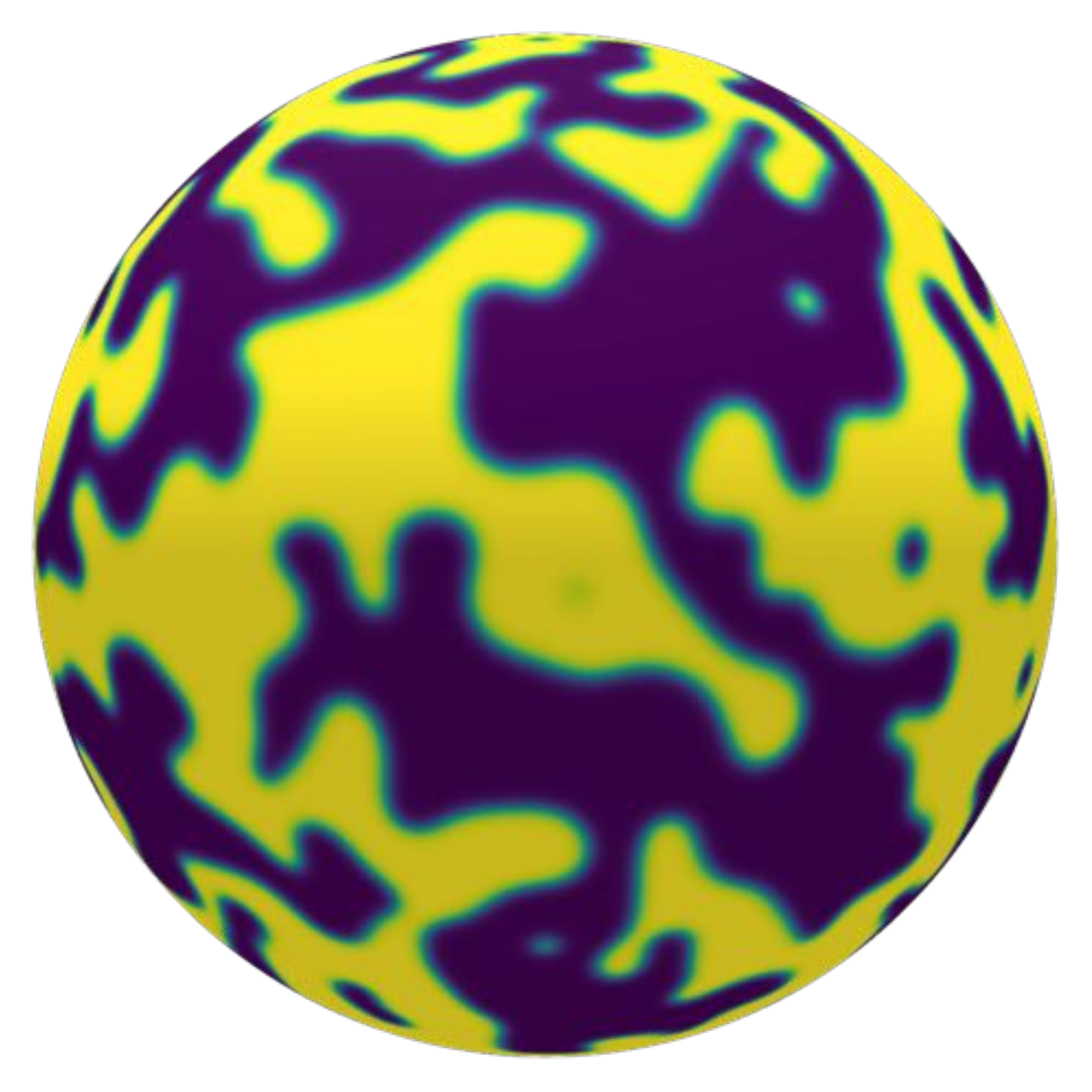}&
\includegraphics[width=0.17\textwidth]{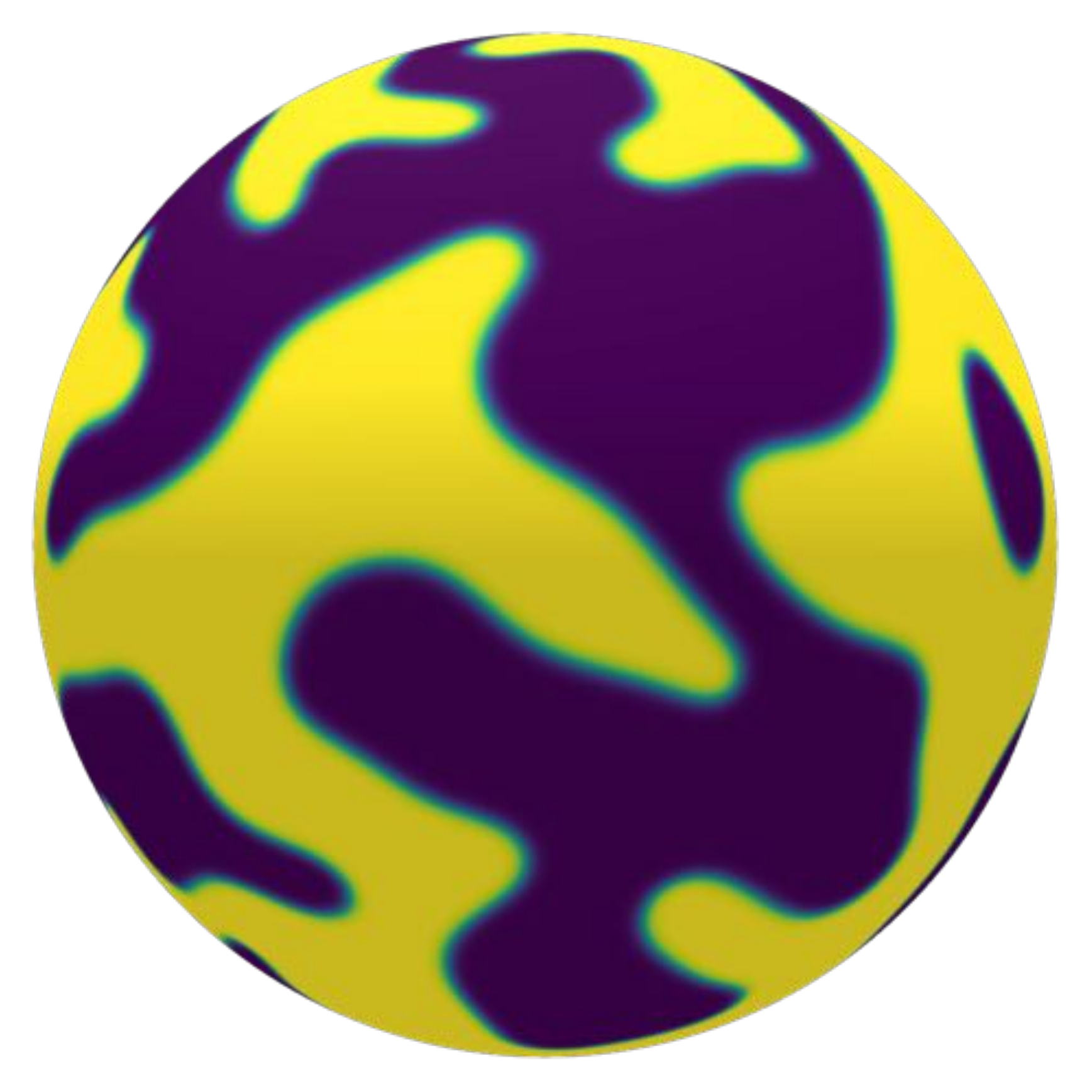}\\
\includegraphics[width=0.17\textwidth]{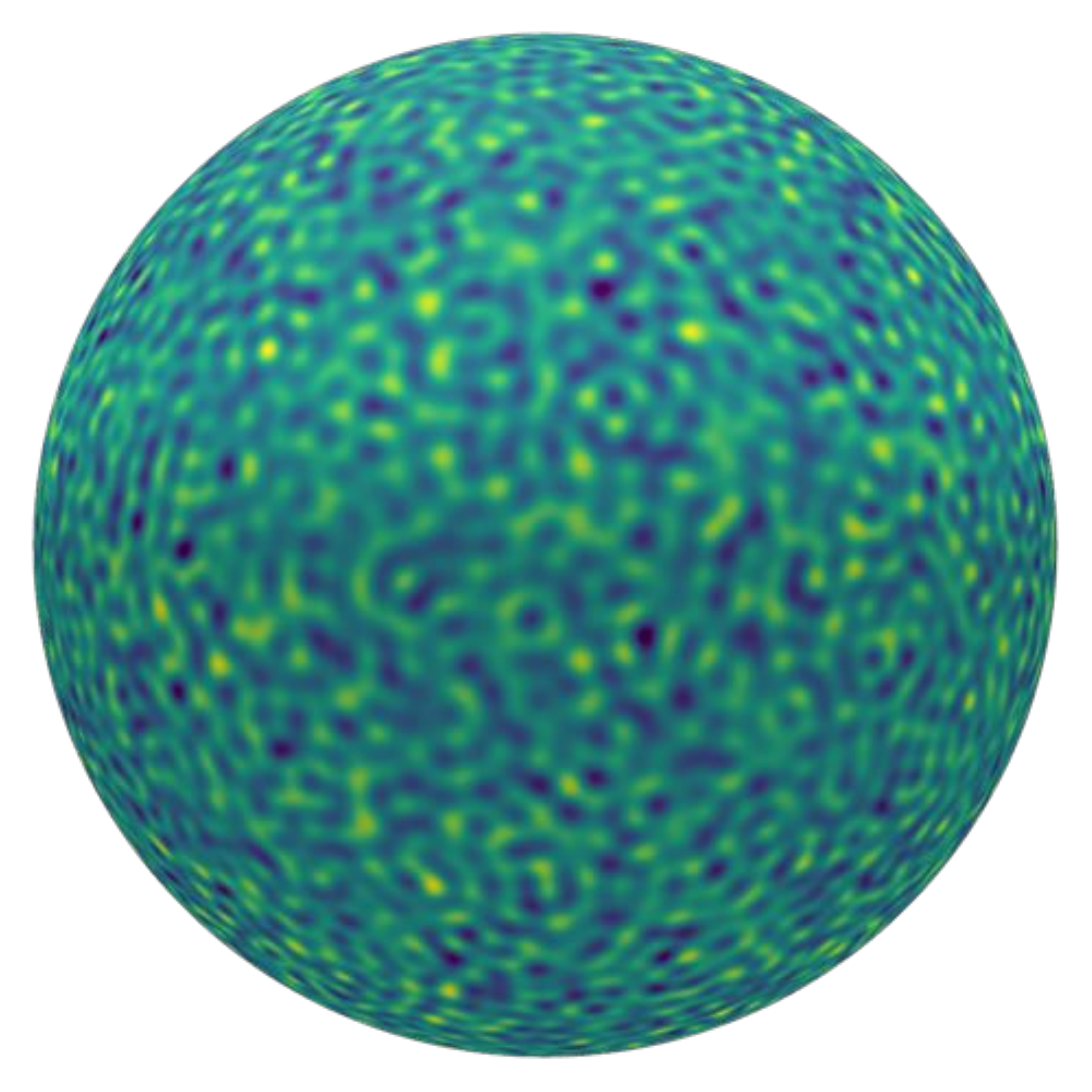}&
\includegraphics[width=0.17\textwidth]{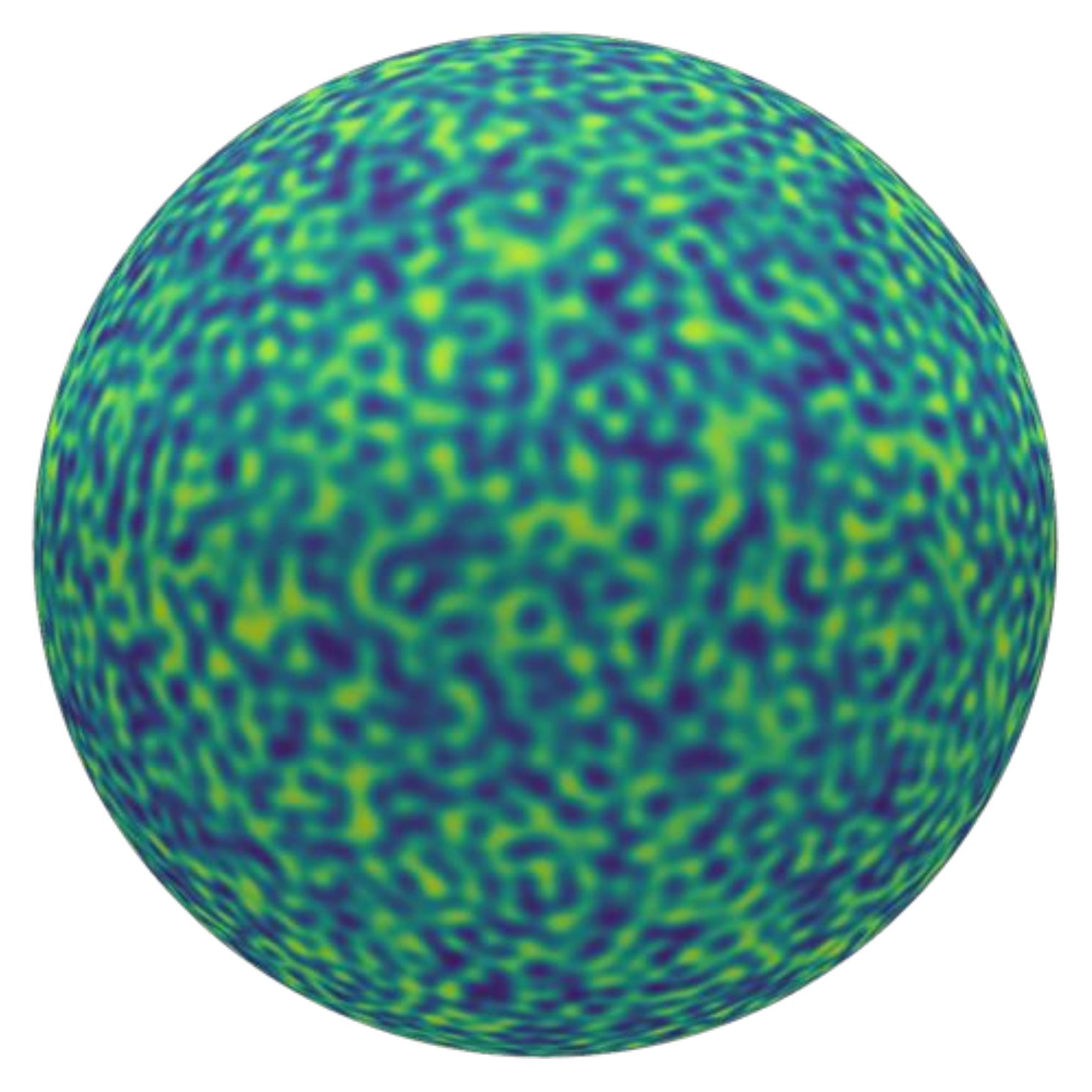}&
\includegraphics[width=0.17\textwidth]{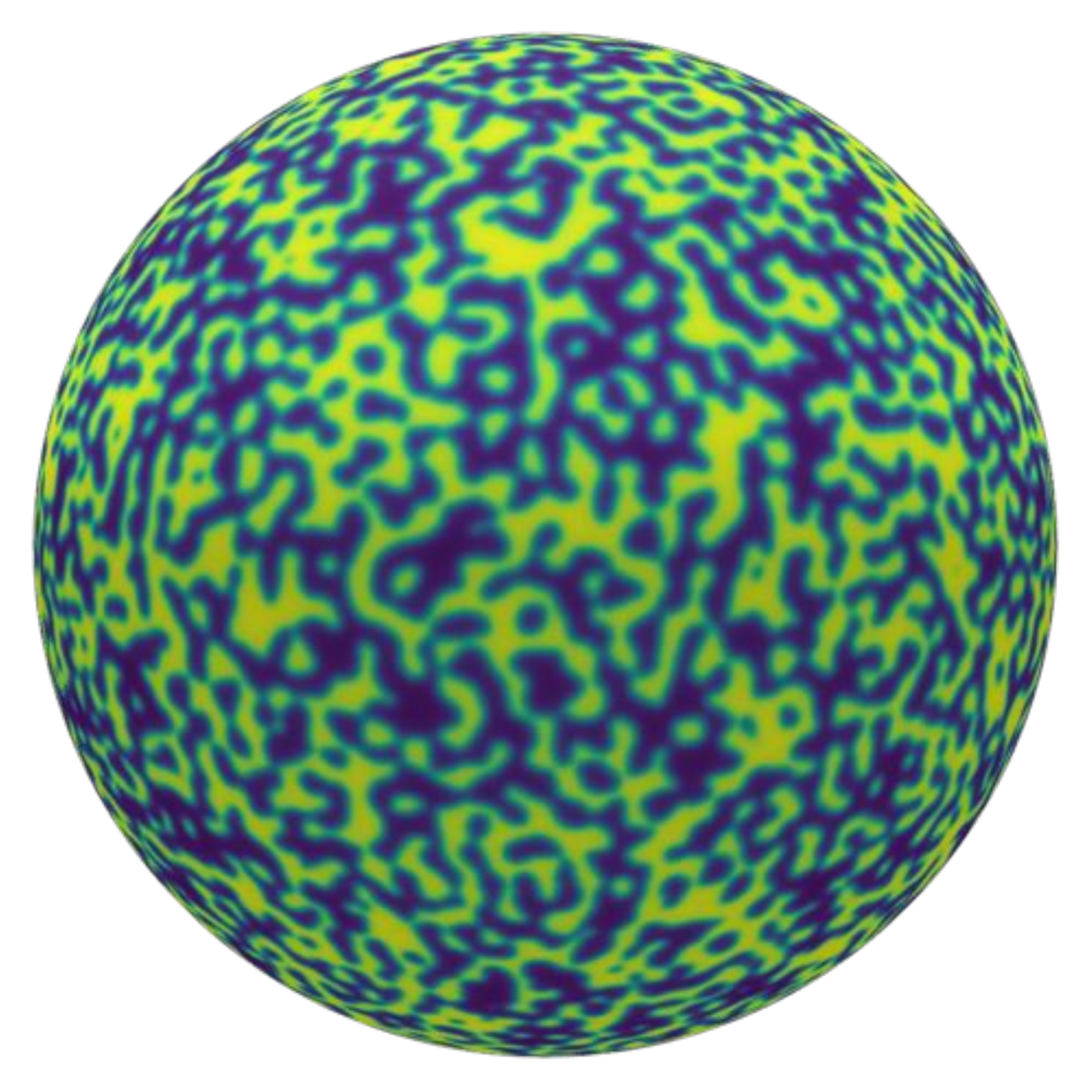}&
\includegraphics[width=0.17\textwidth]{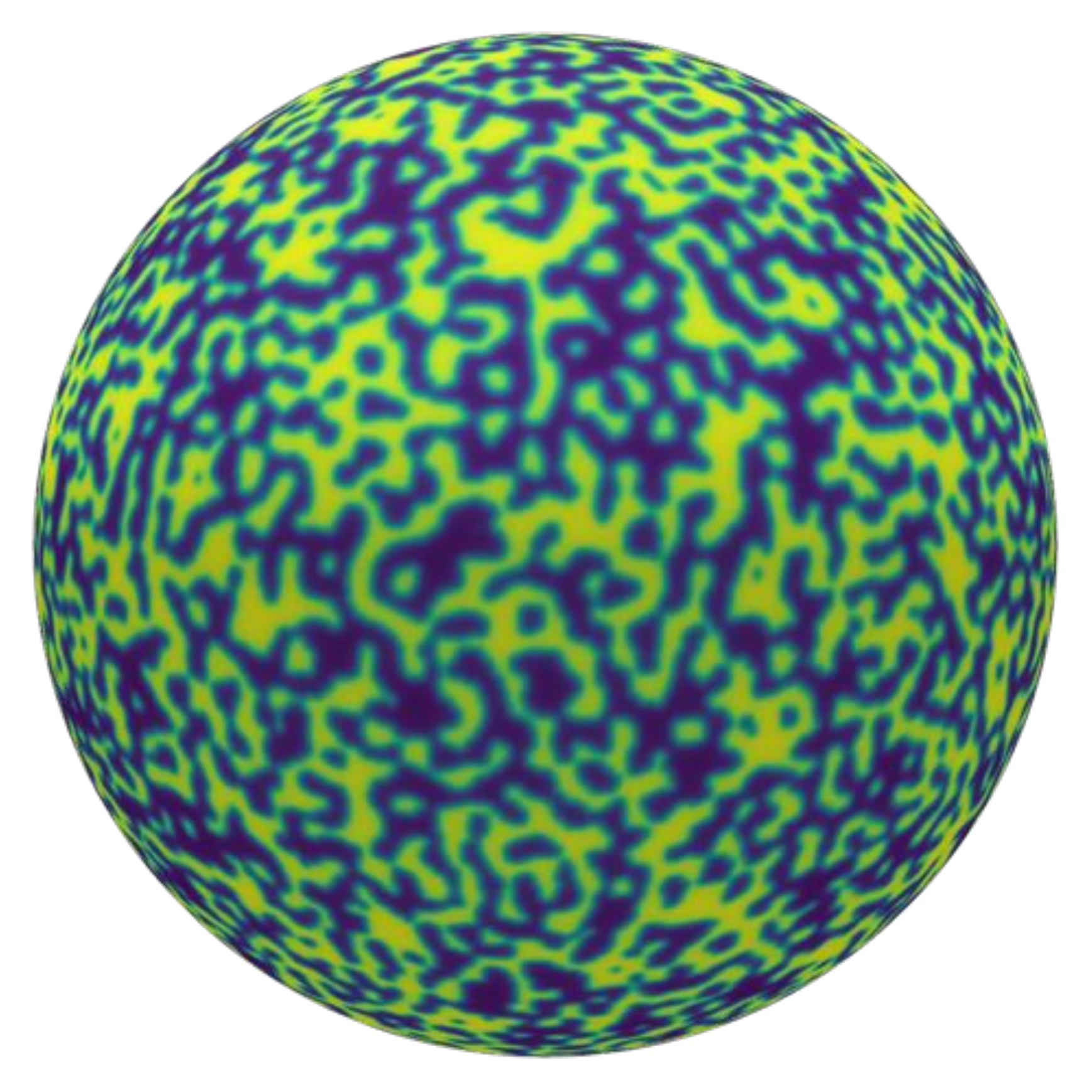}&
\includegraphics[width=0.17\textwidth]{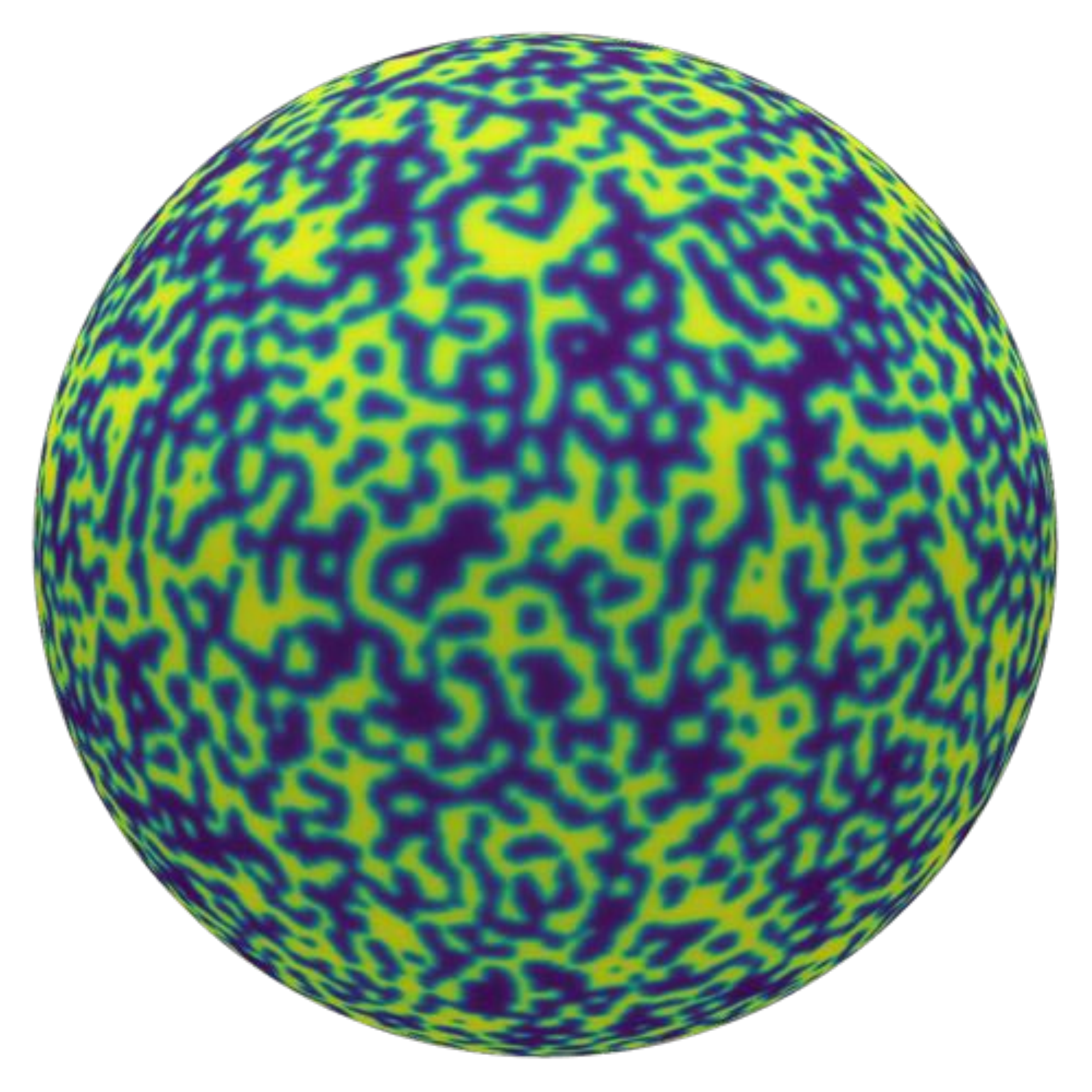}\\
\end{tabular}
\caption{Time evolution of the local (top) and nonlocal (bottom) Allen--Cahn equation starting with a random initial condition provided by a spherical harmonic expansion with the first $128^2$ standard normally distributed numbers seeded by {\tt srand(0)} and scaled by $128^{-1}$ used to populate all coefficients of degree $\le 127$.}
\label{figure:AC}
\end{center}
\end{figure}

To demonstrate spatial convergence, the $2$-norm relative error between the solution obtained with a step size of $h = 2^{-8}$ and varying spherical harmonic degree $n$ 
versus the solution obtained with $h= 2^{-9}$ and $n=1023$ is displayed. 
And for the temporal convergence, the $2$-norm relative error between the solution obtained with a spherical harmonic degree $n = 1023$ and varying step sizes $h$ versus the solution obtained with approximately double the degree, $n=2047$, and half the final step size, $h=2^{-9}$, is displayed. 
Note that due to orthonormality of spherical harmonics, the $2$-norm relative error at $t=T$ between two approximations expanded in spherical harmonics
\begin{equation}
u_1(t=T,\theta,\varphi) = \sum_{\ell=0}^\infty \sum_{m=-\ell}^{+\ell} u_{1,\ell}^m(t=T) Y_\ell^m(\theta,\varphi),
\qquad u_2(t=T,\theta,\varphi) = \sum_{\ell=0}^\infty \sum_{m=-\ell}^{+\ell} u_{2,\ell}^m(t=T) Y_\ell^m(\theta,\varphi),
\end{equation}

\noindent is given by
\begin{equation}
\hbox{Relative Error at $t=T$}
\approx \dfrac{\dsp \sqrt{\sum_{\ell=0}^\infty\sum_{m=-\ell}^{+\ell} \abs{u_{1,\ell}^m(t=T)-u_{2,\ell}^m(t=T)}^2}}
{\dsp \sqrt{\sum_{\ell=0}^\infty\sum_{m=-\ell}^{+\ell} \abs{u_{2,\ell}^m(t=T)}^2}},
\end{equation}

\noindent where we have assumed that $u_2$ is the more accurate approximation to the exact solution.

Figure~\ref{fig:NAC_error} shows temporal and spatial convergences at times $t=1$ and $t=4$. 
Spectral spatial convergence and fourth-order temporal convergences are demonstrated at time $t=1$.
However, the spectral spatial convergence is significantly attenuated by the time $t=4$. By extrapolating from the left panel of Figure~\ref{fig:NAC_error}, it is easy to estimate that it would take a spherical harmonic expansion with a degree well beyond current capabilities to even begin to estimate the true fourth-order temporal convergence at $t=4$. This explains why no temporal convergence is observed at $t=4$. Furthermore, from the convergence estimates at $t=4$, Figure~\ref{fig:NAC_error} invites the possibility that the solution has become discontinuous in finite time.

The qualitative differences in the solutions of the local and nonlocal Allen--Cahn equations are most deftly observed by comparing simulations with random initial conditions. Therefore, in Figure~\ref{figure:AC}, a random initial condition and the solution at a geometrical time progression depict the fundamentally different qualitative behaviors of the local and nonlocal equations. In Figure~\ref{figure:AC}, the same $\epsilon$, $\alpha$, and $\delta$ are used, and the time steps are $h=10^{-1}$ and the maximal spherical harmonic degree is $n=511$.
We also show in Figure~\ref{figure:GL_free_energy} the evolution of the nonlocal Ginzburg--Landau free energy
\begin{equation}
\mathcal{E}(u) = \int_{\Sph^2}\left[-\frac{\epsilon^2}{2}u\LL_\delta u + \frac{1}{4}(u^2-1)^2\right]\ud\Omega. 
\end{equation}

\begin{figure}[t]
\begin{center}
\includegraphics[width=0.75\textwidth, height=0.5\textwidth]{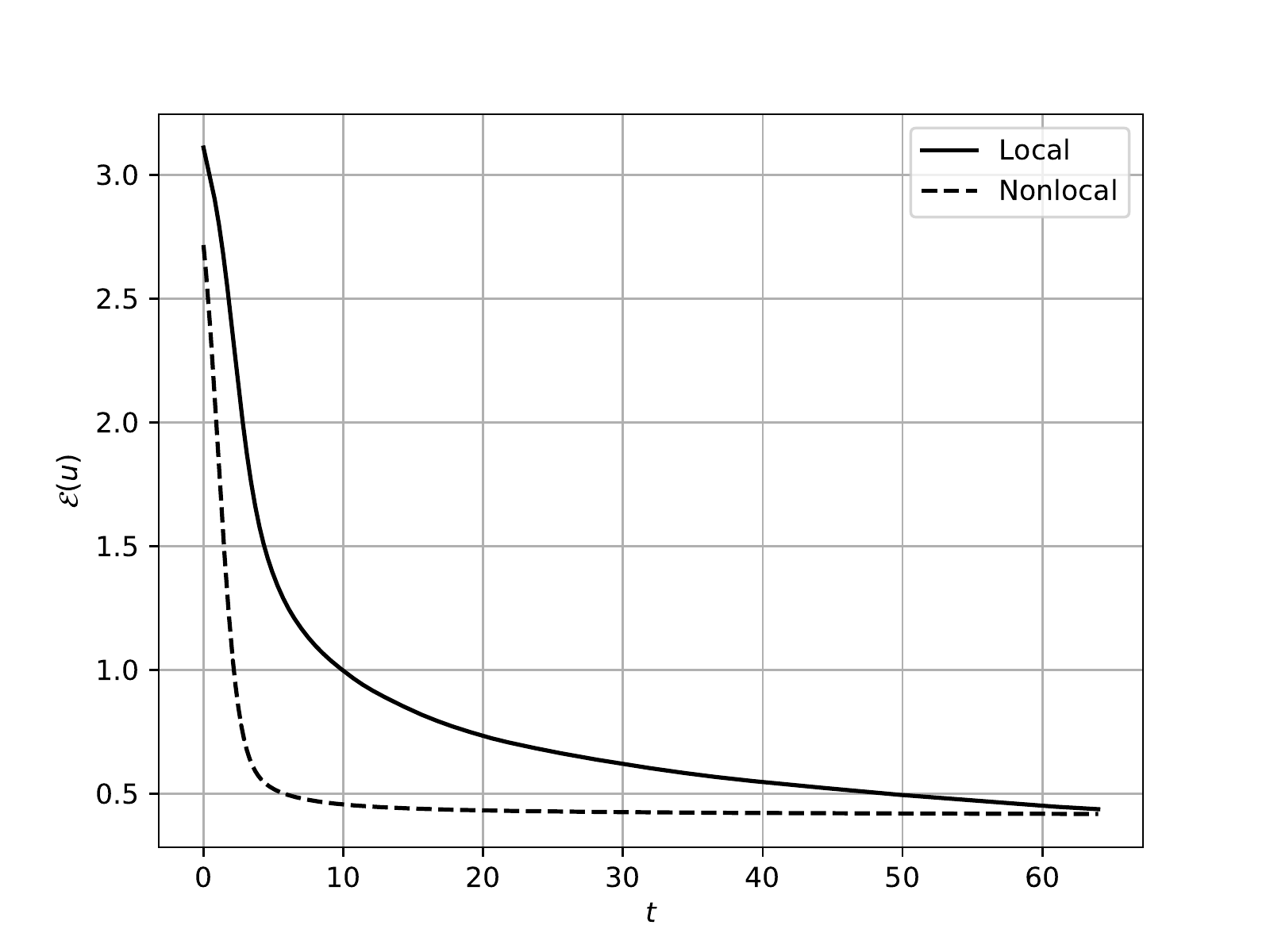}
\caption{Evolution of the Ginzburg--Landau free energies of the local and nonlocal solutions depicted in Figure~\ref{figure:AC}. Both free energies are decreasing functions of time, confirming that the spatial and temporal accuracies are sufficient to exhibit correct qualitative behavior. It appears that the nonlocal free energy may be positively bounded below, while the local free energy appears to continue to decrease.}
\label{figure:GL_free_energy}
\end{center}
\end{figure}

%
%

\paragraph{Brusselator equations}
The dynamics of spot patterns in the Brusselator model for reaction-diffusion systems on the sphere was studied by Trinh and Ward in~\cite{trinh2016}.
A nonlocal version of the Brusselator equations is defined by the coupled system
\begin{equation}
\left\{
\begin{array}{l}
u_t = \epsilon^2\LL_\delta u + \epsilon^2 E -u + fu^2v, \\[10pt]
\tau v_t = \LL_\delta v + \epsilon^{-2}(u-u^2v),
\end{array}
\right. 
\label{eq:BR}
\end{equation}

\noindent for some constants $E>0$, $\epsilon>0$, $\tau>0$, and $0<f<1$. By perturbing the steady-state equilibria
\begin{equation}
u_e = \frac{\epsilon^2E}{1-f} = v_e^{-1},
\end{equation}

\noindent by as little as a $1\%$ fluctuation, the emergence of localized spot patterns is observed~\cite{trinh2016}. 
In Figure~\ref{figure:BR}, we consider a similar random initial condition with the same parameter values $E = 4$, $\epsilon = 0.075$, $\tau = 7.8125$, and $f = 0.8$ 
but also consider the reaction with a nonlocal diffusion. 
In contrast to previous figures, the color is now scaled to the extrema of the solutions $u$ and $v$, with the extrema shown below each plot. 
Whereas we replicate a similar spot pattern for the localized diffusion, when the nonlocal diffusion operator with $\alpha = 0$, $\delta = 1$, $n = 399$, and $h=10^{-1}$, the solution is speckled.

\begin{figure}[t]
\begin{center}
\begin{tabular}{ccccc}
$t=0$ & $t=1$ & $t=4$ & $t=16$ & $t=64$\\
\includegraphics[width=0.17\textwidth]{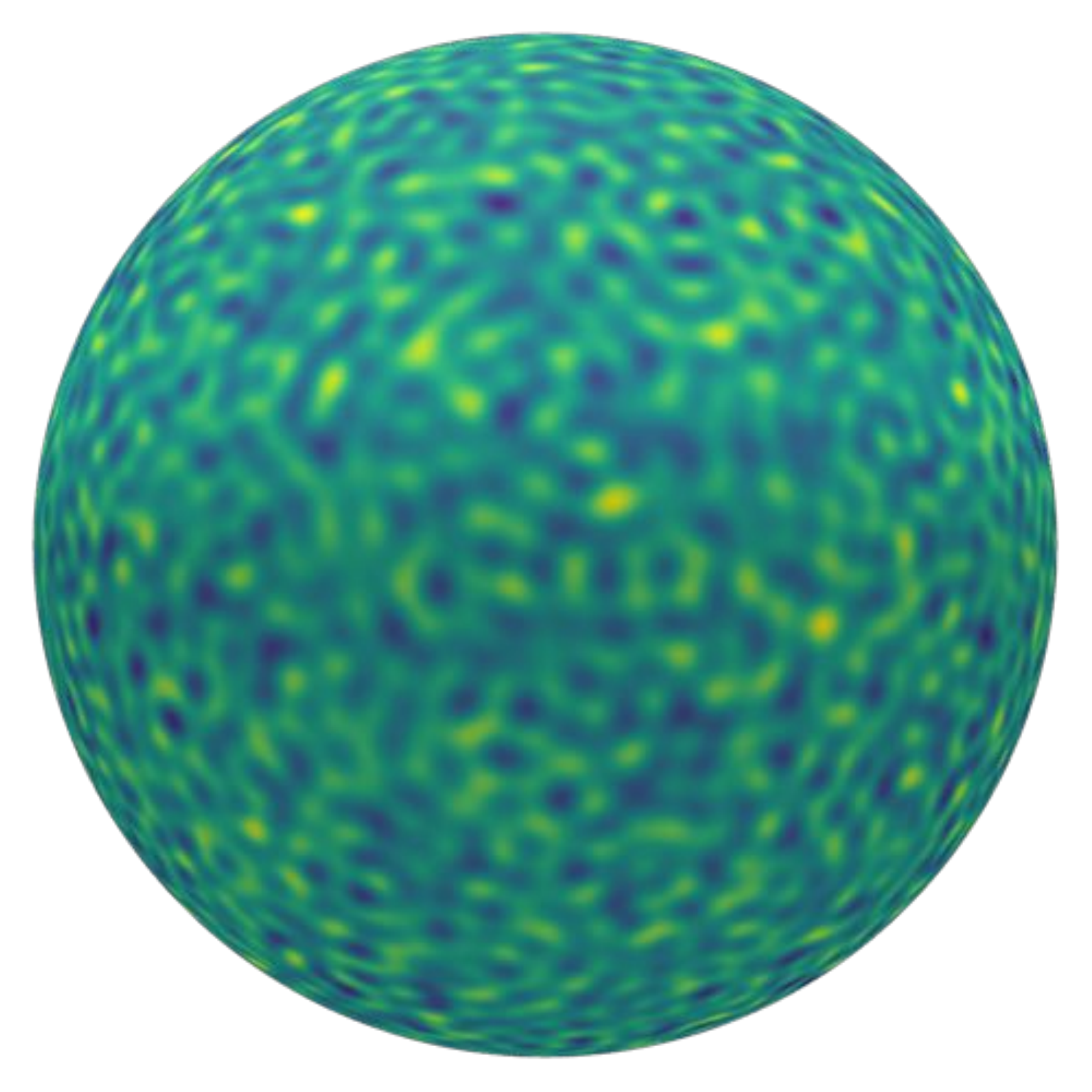}&
\includegraphics[width=0.17\textwidth]{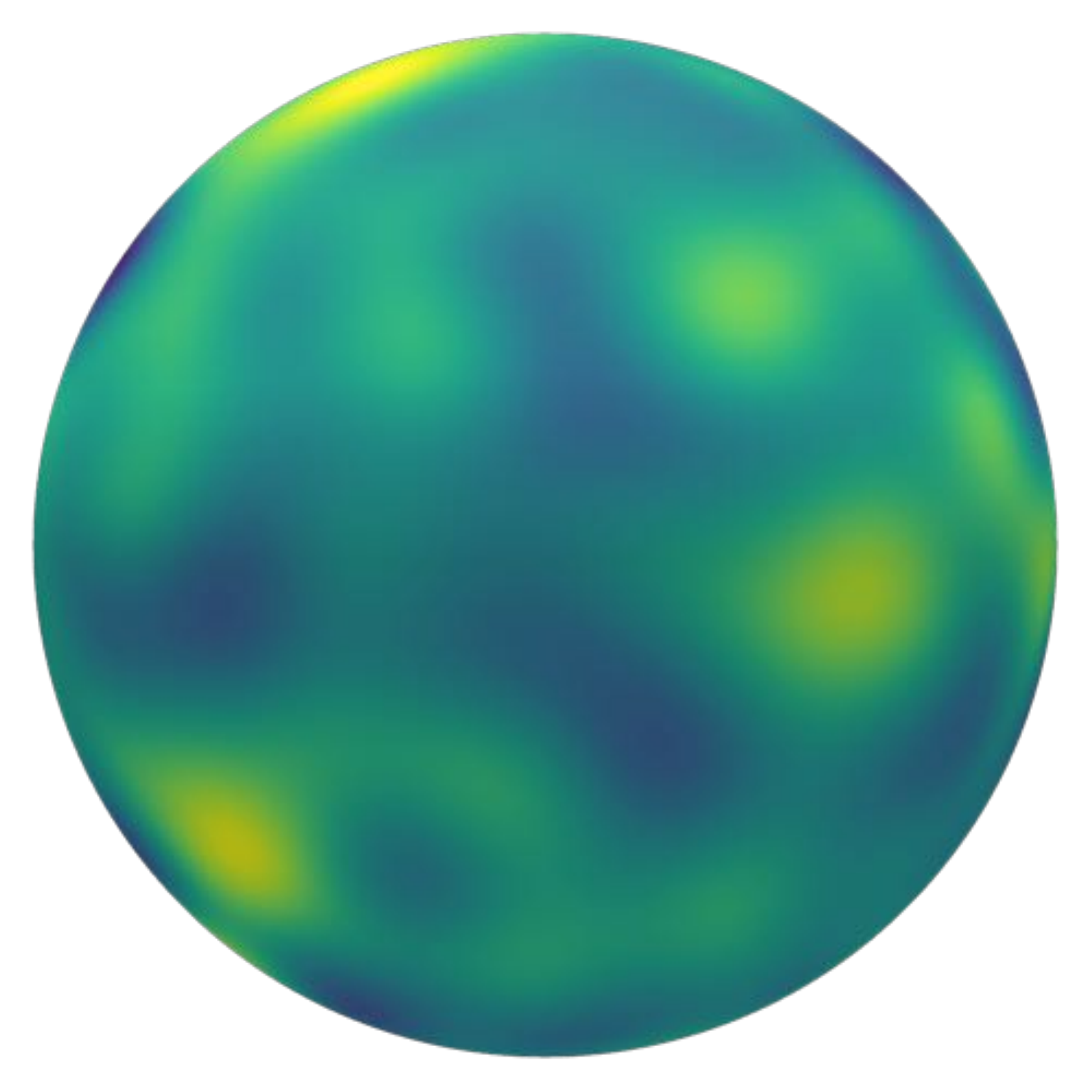}&
\includegraphics[width=0.17\textwidth]{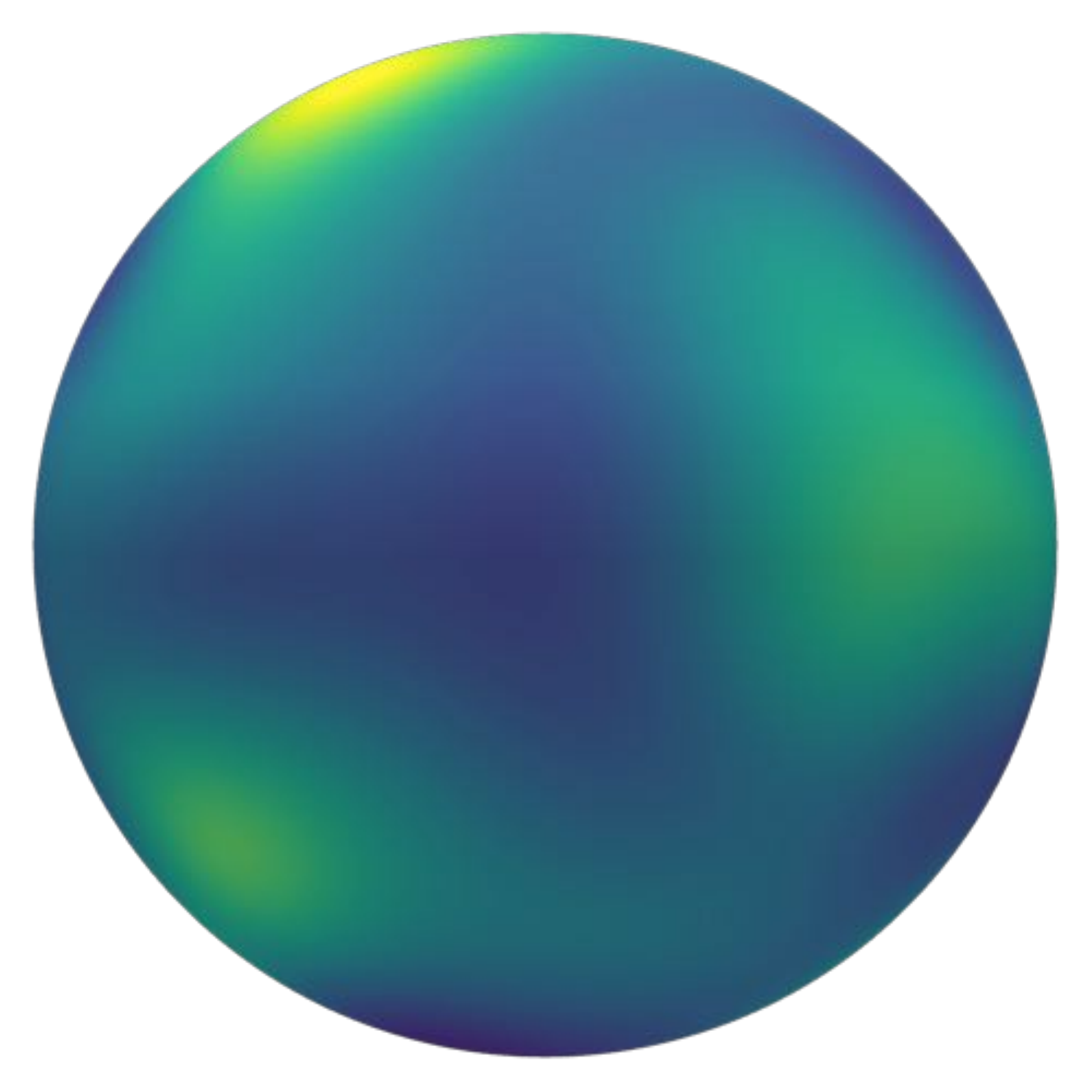}&
\includegraphics[width=0.17\textwidth]{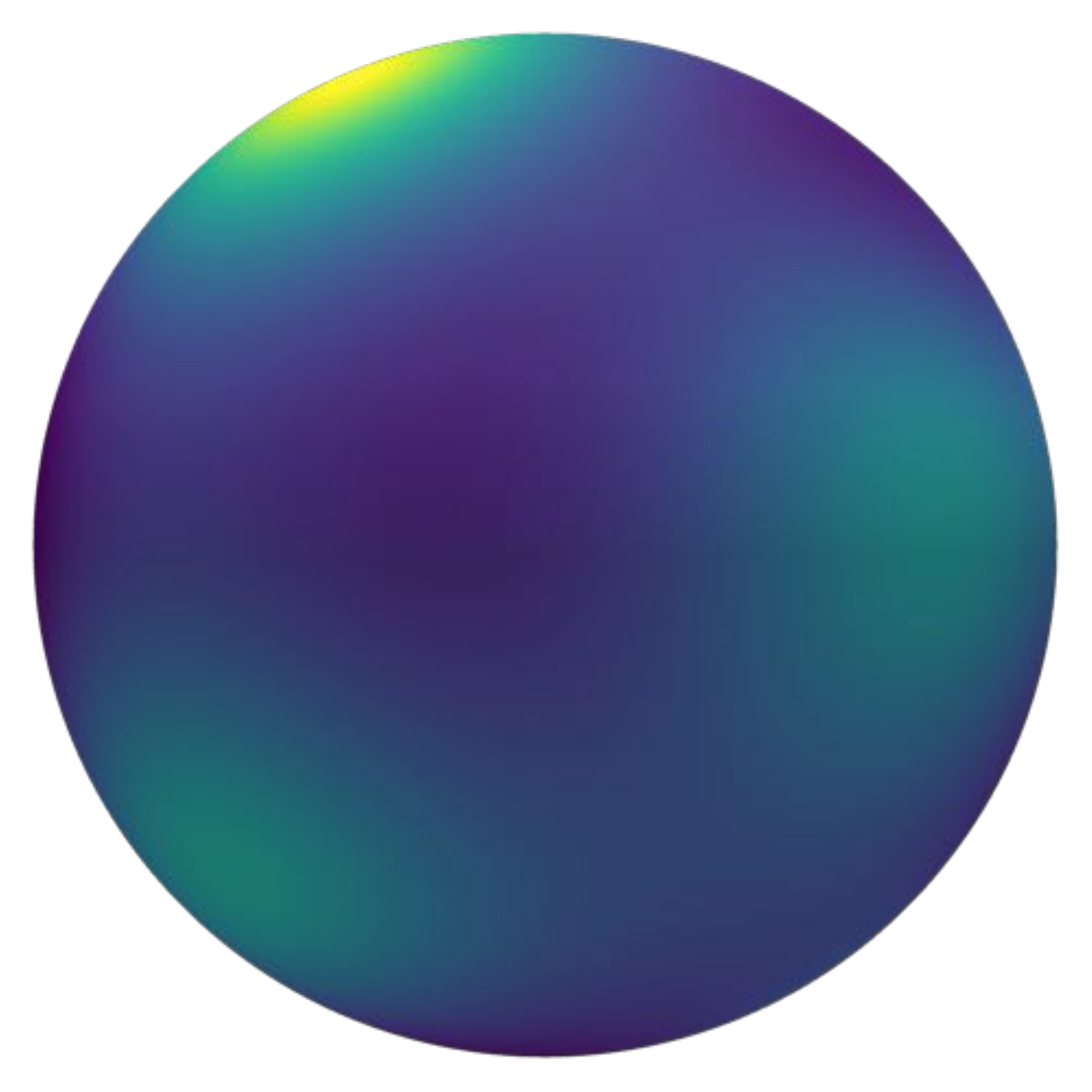}&
\includegraphics[width=0.17\textwidth]{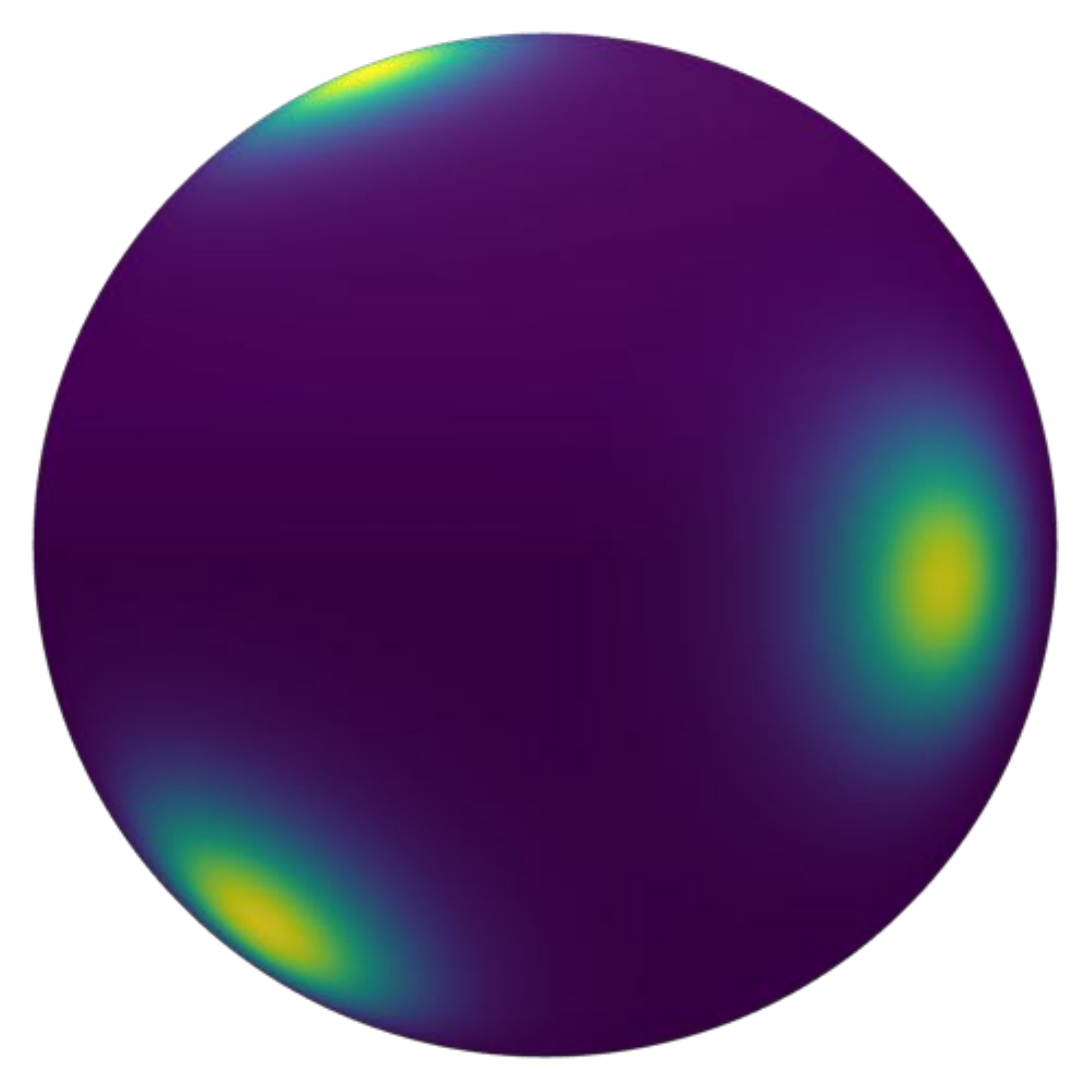}\\
$u\in[0.110,0.114]$ & $u\in[0.112,0.113]$ & $u\in[0.112,0.113]$ & $u\in[0.089,0.173]$ & $u\in[0.024,1.06]$\\
\includegraphics[width=0.17\textwidth]{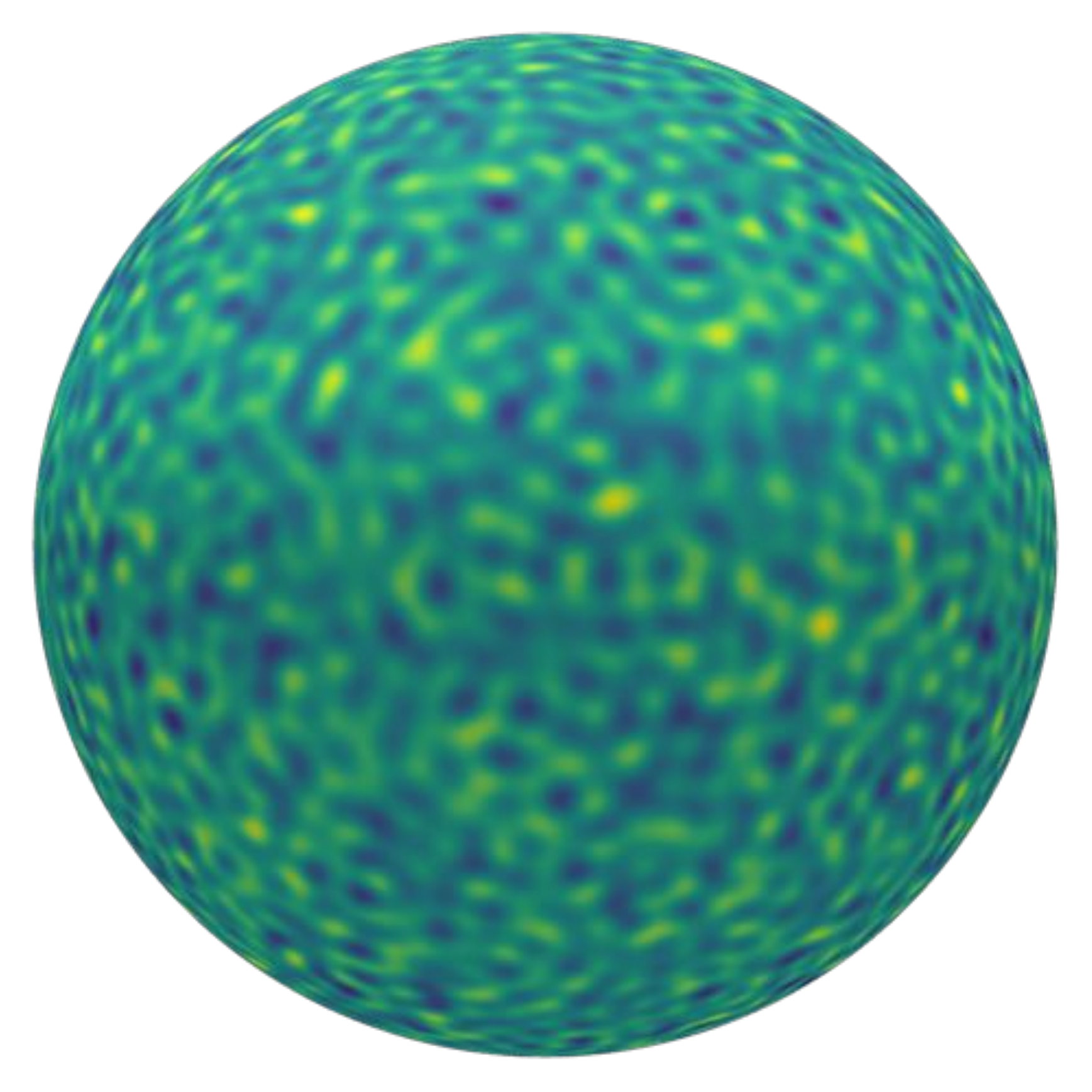}&
\includegraphics[width=0.17\textwidth]{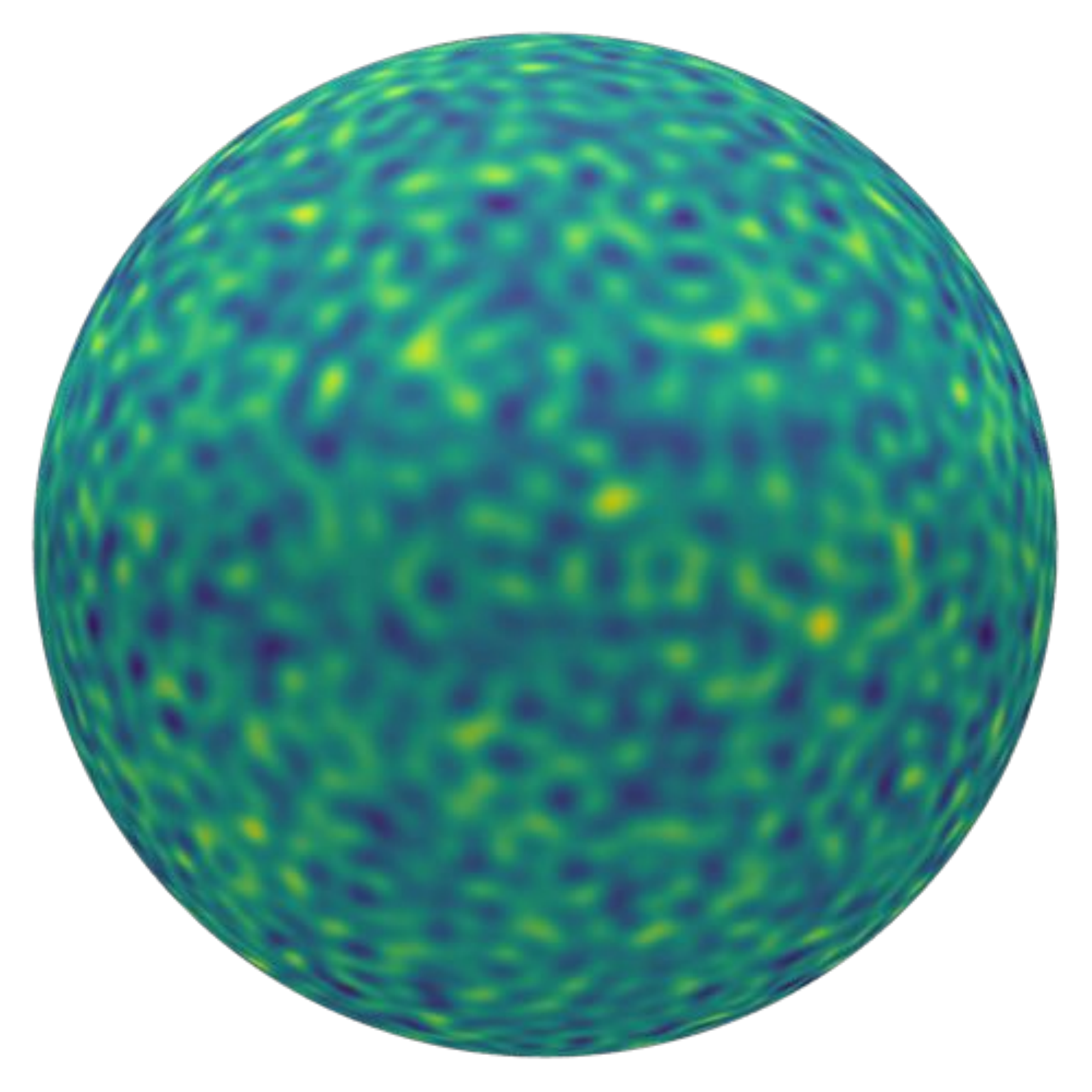}&
\includegraphics[width=0.17\textwidth]{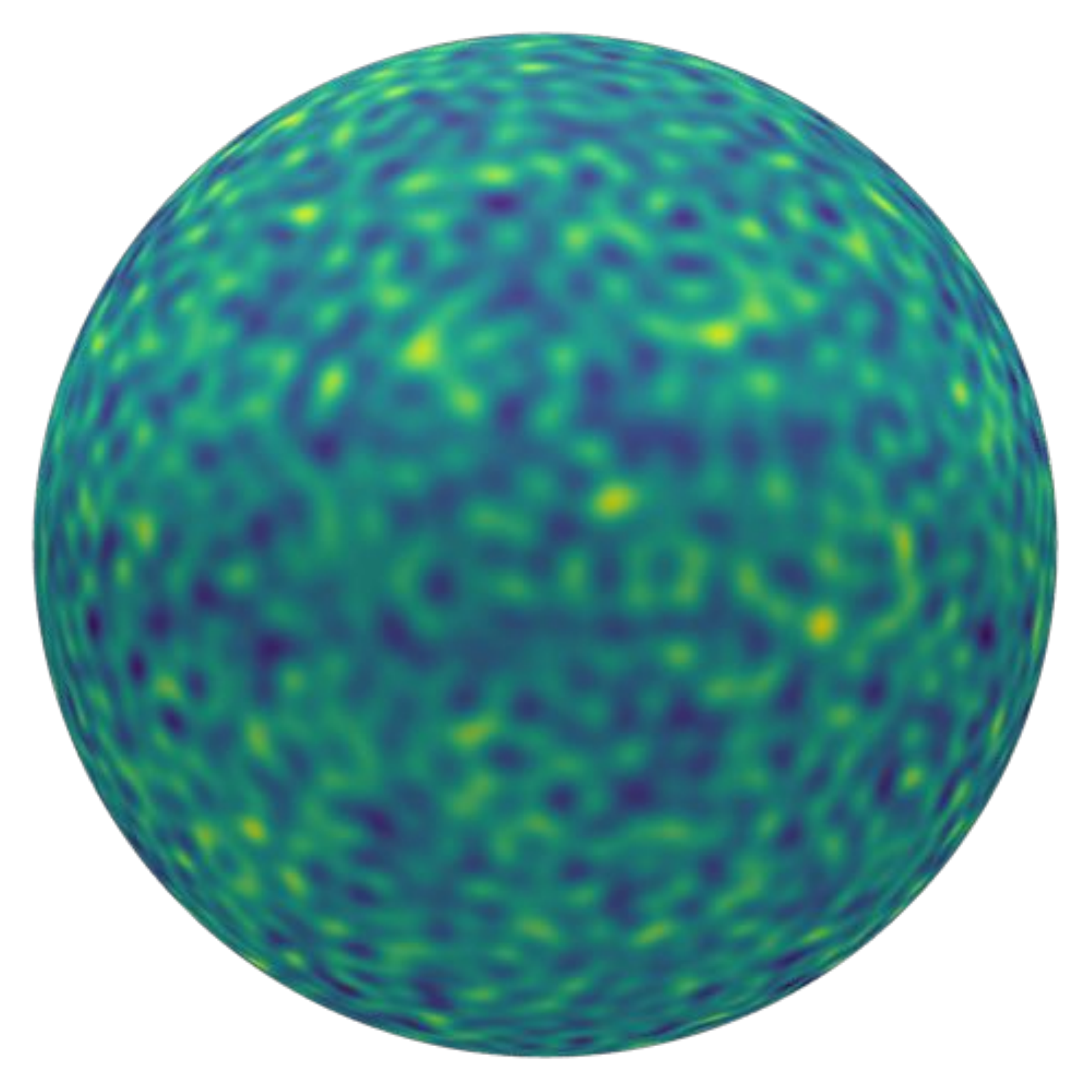}&
\includegraphics[width=0.17\textwidth]{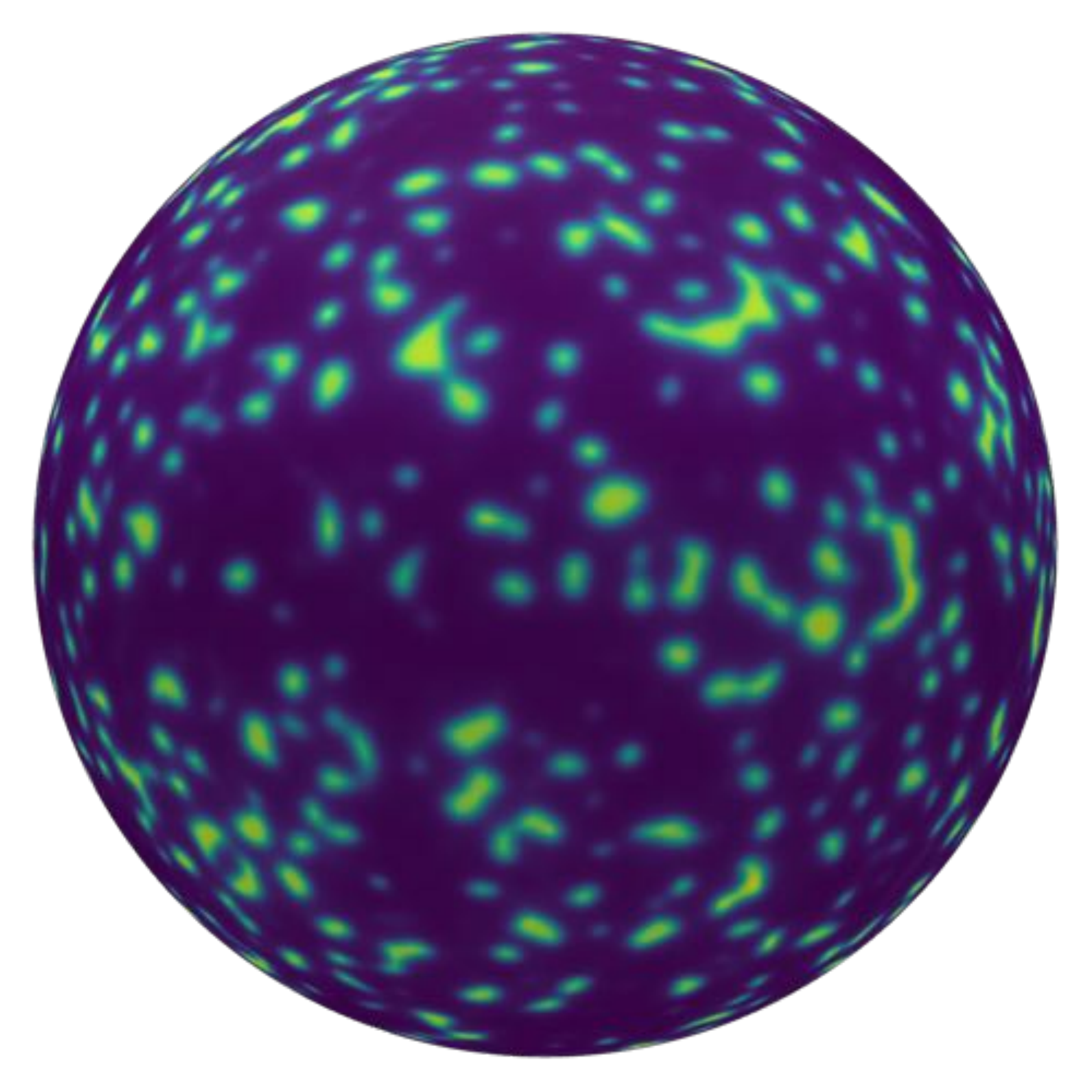}&
\includegraphics[width=0.17\textwidth]{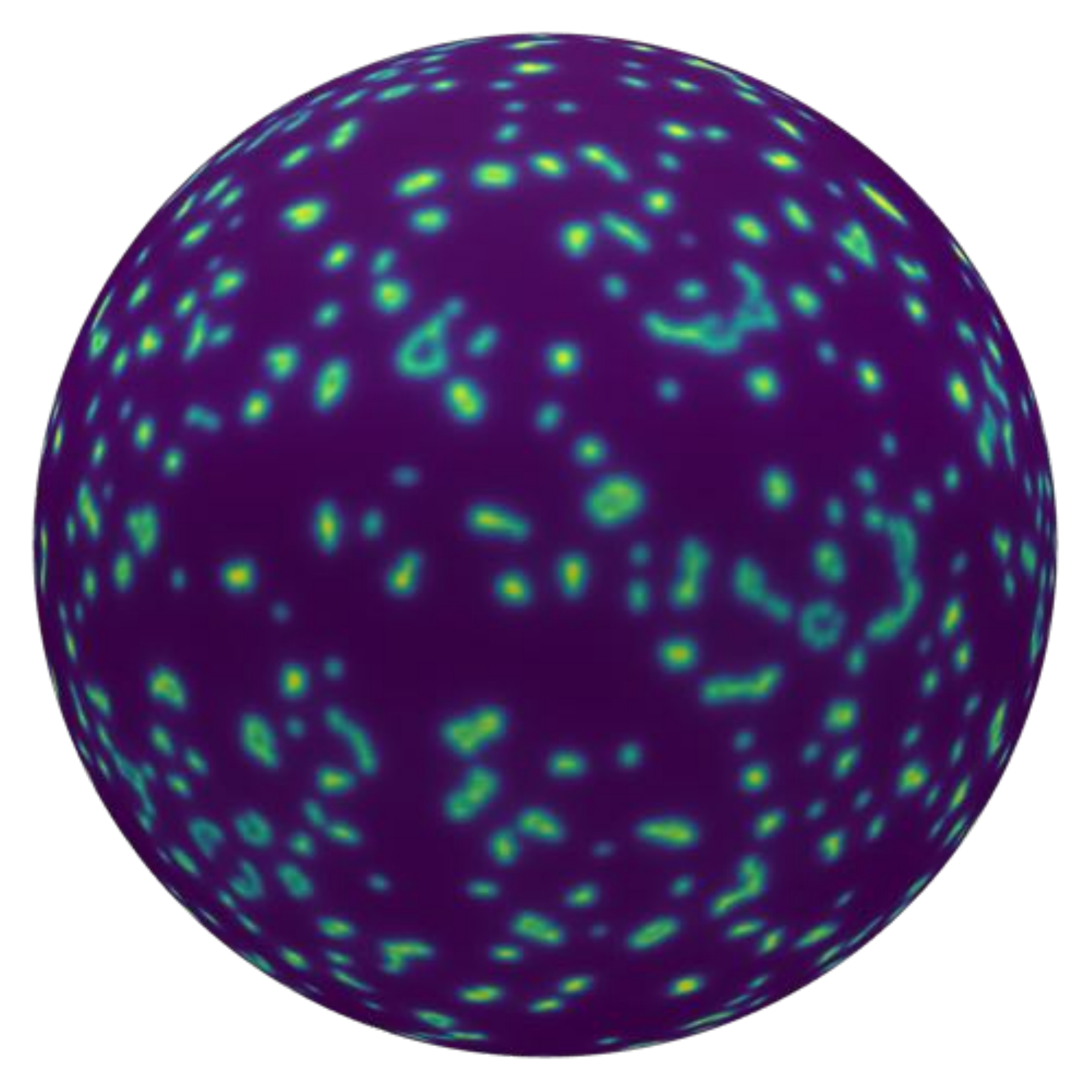}\\
$u\in[0.110,0.114]$ & $u\in[0.110,0.114]$ & $u\in[0.108,0.118]$ & $u\in[0.056,1.05]$ & $u\in[0.036,0.711]$\\
\end{tabular}
\caption{Time evolution of the local (top) and nonlocal (bottom) Brusselator equations with the same parameter values as in \cite{trinh2016} and a similar random initial perturbation of the steady-states provided by a spherical harmonic expansion with the first $2\times 100^2$ standard normally distributed numbers seeded by {\tt srand(0)} and scaled by $(\pi\times 100\times100)^{-1}$ used to populate all coefficients of degree $\le 99$, half for each initial condition.}
\label{figure:BR}
\end{center}
\end{figure}

\section{Discussion}

We have presented algorithms for solving nonlocal diffusion models on the sphere with spectral accuracy in space and high-order
accuracy in time.
These are based on the diagonalization of the nonlocal Laplace--Beltrami operator, the high-accuracy computation of their eigenvalues, 
a fast spherical harmonic transform and exponential integrators.
We have applied our method to the nonlocal Allen--Cahn and Brusselator equations. 
Notwithstanding the potential convergence to discontinuous equilibria accompanied by the Gibbs phenomenon, we are able to remove the non-physical oscillations by the use of Ces\`aro means.

Our {\sc Julia} codes are available online at {\em GitHub} (\texttt{FastTransforms.jl} and \texttt{SpectralTimeStepping.jl} packages).
A MATLAB version, based on Chebfun \cite{chebfun} and its recent extensions to periodic problems \cite{montanelli2015b} and 
the sphere \cite{montanelli2017b, townsend2016} and {\tt mex}-ing in the fast spherical harmonic transform from {\sc Julia}, is available upon request.

There are many ways to continue the analysis of nonlocal diffusion operators on the sphere.
Nonlocal diffusion operators introduce new qualitative behaviors in contrast to classical PDEs of evolution and their steady-states. Indeed, nonlocal interactions are ubiquitous in nature and they are also generic features of model reduction~\cite{du2018}.

Our first particular choice of a nonlocal diffusion operator \eqref{operator} is motivated by an expedited numerical analysis of the spectrum. Another choice for a nonlocal operator is the fractional Laplacian,\footnote{It can be shown that the fractional Laplacian corresponds to our nonlocal operator with particular choices of the parameters $\alpha$ and $\delta$ for which the numerical integration may be performed exactly.} and yet another reasonable modification is to use geodesic distance in place of Euclidean distance. In \cite{slevinsky-olver2017}, a general approach is used for singular integral equations that could well be adapted to more general kernels $\rho_\delta(\xb, \yb)$ in the present setting. While beyond the scope of this report, one may argue that the nonlocal diffusion operator \eqref{operator} provides many of the qualitative behaviors that one might expect, has a natural integral definition that may be related to measurable physical quantities, and also has a few parameters that may be useful for modelling purposes. Furthermore, the Euclidean and geodesic distances are isomorphic and a function of either one may be rapidly expanded as a series of functions of the other.

While the steady-state of the nonlocal Allen--Cahn equation is known to be  possibly discontinuous \cite{du2016}, 
a potential topic for the analysis of nonlocal operators is whether discontinuities are attained in finite or infinite time. 
The spatial convergence of spectral methods, such as ours, assimilate the regularity of the solution: for analytic or entire solutions, this property makes them extremely competitive, structured linear algebra permitting; for discontinuous solutions, Ces\`aro means remove the Gibbs phenomenon, but it is not clear that a spectral method may perform any better or worse than a finite difference/element/volume method. An advantage of the spherical harmonic basis is the trivial computation of the spectrum of the nonlocal operator. If another spatial discretization were chosen, it is possible for this to become the bottleneck in the simulation of the evolution of the dynamical system. Another advantage is that we expect spectral accuracy {\em arbitrarily close to} the discontinuous steady-state if the initial condition is sufficiently smooth.

Other extensions include the design of similar algorithms for higher-order diffusion or gradient-type operators,
and the numerical solution of nonlocal PDEs on spheroids and other manifolds.
Nonlocal phase-field crystal models on the sphere could also be an exciting application.

\section*{Acknowledgments}

We thank the members of the CM3 group (Ran Gu, Hwi Lee, Qi Sun and Yunzhe Tao) at Columbia University 
for fruitful discussions. We thank Feng Dai for an enlightening discussion on Ces\`aro means. We also thank Nick Trefethen for introducing us to Ignace Bogaert, who so effectively demonstrates the utility of Szeg\H o's asymptotic formula for iteration-free Gauss--Legendre quadrature. 
Finally, we would like to express our gratitude to the reviewers (Alex Townsend and the anonymous one) for their constructive reports.
Their comments and suggestions helped us to significantly improve the quality of the manuscript.
This research is supported in part by NSERC RGPIN-2017-05514, NSF DMS-1719699, AFOSR MURI center for material failure prediction through peridynamics, and the ARO MURI Grant W911NF-15-1-0562.

\bibliography{references}

\begin{thebibliography}{10}
\expandafter\ifx\csname url\endcsname\relax
  \def\url#1{\texttt{#1}}\fi
\expandafter\ifx\csname urlprefix\endcsname\relax\def\urlprefix{URL }\fi
\expandafter\ifx\csname href\endcsname\relax
  \def\href#1#2{#2} \def\path#1{#1}\fi

\bibitem{bates1999}
P.~W. Bates, A.~Chmaj, An integrodifferential model for phase transitions:
  stationary solutions in higher space dimensions, J. Stat. Phys. 95 (1999)
  1119--1139.

\bibitem{bobaru2010}
F.~Bobaru, M.~Duangpanya, The peridynamic formulation for transient heat
  conduction, Int. J. Heat Mass Tranf. 53 (2010) 4047--4059.

\bibitem{du2017a}
Q.~Du, J.~Yang, Fast and accurate implementation of {F}ourier spectral
  approximations of nonlocal diffusion operators and its applications, J.
  Comput. Phys 332 (2017) 118--134.

\bibitem{gilboa2008}
G.~Gilboa, S.~Osher, Nonlocal operators with applications to image processing,
  Multiscale Model. Simul. 7~(3) (2008) 1005--1028.

\bibitem{kao2010}
C.-Y. Kao, Y.~Lou, W.~Shen, Random dispersal vs. non-local dispersal, Discrete
  Contin. Dyn. Syst. 26~(2) (2010) 551--596.

\bibitem{silling2000}
S.~A. Silling, Reformulation of elasticity theory for discontinuities and
  long-range forces, J. Mech. Phys. Solids 48 (2000) 175--209.

\bibitem{du2012}
Q.~Du, M.~Gunzburger, R.~B. Lehoucq, K.~Zhou, Analysis and approximation of
  nonlocal diffusion problems with volume constraints, SIAM Rev. 54 (2012)
  667--696.

\bibitem{du2013}
Q.~Du, M.~Gunzburger, R.~B. Lehoucq, K.~Zhou, A nonlocal vector calculus,
  nonlocal volume-constrained problems, and nonlocal balance laws, Math. Models
  Methods Appl. Sci. 23 (2013) 493--540.

\bibitem{du2017b}
Q.~Du, Local limits and asymptotically compatible discretizations, in: Handbook
  of peridynamic modeling, Adv. Appl. Math., CRC Press, Boca Raton, 2017, pp.
  87--108.

\bibitem{du2018}
Q.~Du, Nonlocal modeling, analysis and computation, NSF-CBMS Monograph, SIAM,
  Philadelphia, 2018.

\bibitem{du2016}
Q.~Du, J.~Yang, Asymptotically compatible {F}ourier spectral approximations of
  nonlocal {A}llen--{C}ahn equations, SIAM J. Numer. Anal. 54 (2016)
  1899--1919.

\bibitem{atkinson2012}
K.~Atkinson, W.~Han, {Spherical Harmonics and Approximations on the Unit
  Sphere: An Introduction}, Springer, Berlin, 2012.

\bibitem{merilees1973}
P.~E. Merilees, The pseudospectral approximation applied to the shallow water
  equations on a sphere, Atmosphere 11 (1973) 13--20.

\bibitem{orszag1973}
S.~A. Orszag, Fourier series on spheres, Mon. Wea. Rev. 102 (1974) 56--75.

\bibitem{montanelli2017phd}
H.~Montanelli, Numerical algorithms for differential equations with
  periodicity, Ph.D. thesis, University of Oxford (2017).

\bibitem{montanelli2017b}
H.~Montanelli, Y.~Nakatsukasa, Fourth-order time-stepping for stiff {PDEs} on
  the sphere, arXiv:1701.06030, 2017.

\bibitem{townsend2016}
A.~Townsend, H.~Wilber, G.~B. Wright, Computing with functions in spherical and
  polar geometries, {I. The} sphere, SIAM J. Sci. Comput. 38 (2016) C403--C425.

\bibitem{rokhlin2006}
V.~Rokhlin, M.~Tygert, Fast algorithms for spherical harmonic expansions, SIAM
  J. Sci. Comput. 27 (2006) 1903--1928.

\bibitem{tygert2008}
M.~Tygert, Fast algorithms for spherical harmonic expansions, {II}, J. Comput.
  Phys. 227 (2008) 4260--4279.

\bibitem{tygert2010}
M.~Tygert, Fast algorithms for spherical harmonic expansions, {III}, J. Comput.
  Phys. 229 (2010) 6181--6192.

\bibitem{slevinsky2017a}
R.~M. Slevinsky, Fast and backward stable transforms between spherical harmonic
  expansions and bivariate {F}ourier series, Appl. Comput. Harmon. Anal.

\bibitem{slevinsky2017b}
R.~M. Slevinsky, Conquering the pre-computation in two-dimensional harmonic
  polynomial transforms, arXiv:1711.07866 (2017).

\bibitem{cox2002}
S.~M. Cox, P.~C. Matthews, Exponential time differencing for stiff systems, J.
  Comput. Phys. 176 (2002) 430--455.

\bibitem{du2014}
X.~Tian, Q.~Du, Analysis and comparison of different approximations to nonlocal
  diffusion and linear peridynamic equations, SIAM J. Numer. Anal. 51~(6)
  (2013) 3458--3482.

\bibitem{trefethen2000}
L.~N. Trefethen, Spectral Methods in MATLAB, SIAM, Philadelphia, 2000.

\bibitem{canuto2007}
C.~Canuto, M.~Y. Hussaini, A.~Quarteroni, T.~A. Zang, Spectral Methods:
  Evolution to Complex Geometries and Applications to Fluid Dynamics, Springer,
  Berlin, 2007.

\bibitem{shen2011}
J.~Shen, T.~Tang, L.-L. Wang, Spectral Methods: Algorithms, Analysis and
  Applications, Springer, Berlin, 2011.

\bibitem{szego1934}
G.~Szeg\H{o}, \"{U}ber einige asymptotische entwicklungen der {Legendreschen}
  funcktionen, Proc. Lond. Math. Soc. 50 (1934) 427--450.

\bibitem{sommariva2013}
A.~Sommariva, Fast construction of {F}ej\'er and {C}lenshaw--{C}urtis rules for
  general weight functions, Comp. Math. Appl. 65 (2013) 682--693.

\bibitem{waldvogel2006}
J.~Waldvogel, Fast construction of the {F}ej\'er and {C}lenshaw--{C}urtis
  quadrature rules, BIT Numer. Math. 46 (2006) 195--202.

\bibitem{abramowitz1965}
M.~Abramowitz, I.~A. Stegun, Handbook of Mathematical Functions with Formulas,
  Graphs, and Mathematical Tables, Dover Publications, New York, 1965.

\bibitem{piessens1987}
R.~Piessens, Modified {Clenshaw--Curtis} integration and applications to
  numerical computation of integral transforms, in: P.~Keast, G.~Fairweather
  (Eds.), Numerical Integration: Recent Developments, Software and
  Applications, Springer Netherlands, Dordrecht, 1987, pp. 35--51.

\bibitem{olver2010}
F.~W.~J. Olver, D.~W. Lozier, R.~F. Boisvert, C.~W. Clark, NIST Handbook of
  Mathematical Functions, Cambridge University Press, Cambridge, 2010.

\bibitem{rainville1960}
E.~Rainville, Special Functions, MacMillan, New York, 1960.

\bibitem{xiang2014}
S.~Xiang, G.~He, H.~Wang, On fast and stable implementation of
  {C}lenshaw--{C}urtis and {F}ej\'er-type quadrature rules, Abst. Appl. Anal.
  2014.

\bibitem{bogaert2014}
I.~Bogaert, Iteration-free computation of {Gauss--Legendre} quadrature nodes
  and weights, SIAM J. Sci. Comput. 36 (2014) A1008--A1026.

\bibitem{montanelli2017a}
H.~Montanelli, N.~Bootland, {Solving periodic semilinear stiff PDEs in
  $1\mrm{D}$, $2\mrm{D}$ and $3\mrm{D}$ with exponential integrators},
  arXiv:1604.08900, 2016.

\bibitem{kassam2005}
A.-K. Kassam, L.~N. Trefethen, Fourth-order time-stepping for stiff {PDE}s,
  SIAM J. Sci. Comput. 26 (2005) 1214--1233.

\bibitem{du2005}
Q.~Du, W.~Zhu, Analysis and applications of the exponential time differencing
  schemes and their contour integral modifications, BIT Numer. Math. 45 (2005)
  307--328.

\bibitem{zygmund1959}
A.~Zygmund, Trigonometric Series, Cambridge University Press, Cambridge, 1959.

\bibitem{weyl1910}
H.~Weyl, {Die Gibbssche Erscheinung in der Theorie der Kugelfunktionen}, Rend.
  Circ. Matem. Palermo 29 (1910) 308--323.

\bibitem{dai2013}
F.~Dai, Y.~Xu, Approximation Theory and Harmonic Analysis on Spheres and Balls,
  Springer, New York, 2013.

\bibitem{gelb1997}
A.~Gelb, The resolution of the {Gibbs} phenomenon for spherical harmonics,
  Math. Comp. 66 (1997) 699--717.

\bibitem{allen1979}
S.~M. Allen, J.~W. Cahn, A microscopic theory for antiphase boundary motion and
  its application to antiphase domain coarsening, Acta Metall. 27 (1979)
  1085--1095.

\bibitem{du2008}
Y.~Du, The heterogeneous {Allen--Cahn} equation in a ball: Solutions with
  layers and spikes, J. Differential Equations 244 (2008) 117--169.

\bibitem{trinh2016}
P.~H. Trinh, M.~J. Ward, The dynamics of localized spot patterns for
  reaction-diffusion systems on the sphere, Nonlinearity 29 (2016) 766--806.

\bibitem{chebfun}
T.~A. Driscoll, N.~Hale, L.~N. Trefethen (Eds.), {C}hebfun {G}uide, Pafnuty
  Publications, Oxford, 2014; see also \url{www.chebfun.org}.

\bibitem{montanelli2015b}
G.~B. Wright, M.~Javed, H.~Montanelli, L.~N. Trefethen, Extension of {C}hebfun
  to periodic functions, SIAM J. Sci. Comput. 37 (2015) C554--C573.

\bibitem{slevinsky-olver2017}
R.~M. Slevinsky, S.~Olver, A fast and well-conditioned spectral method for
  singular integral equations, J. Comp. Phys. 332 (2017) 290--315.

\end{thebibliography}

\appendix

\renewcommand*{\thesection}{\Alph{section}}
\section{Proof of the generalized Funk--Hecke formula}\label{appendix:GFH}

\begin{lemma}\label{lemma:GFHlhs}
Provided $t\mapsto(1-t)f(t)\in L^1([-1,1])$,
\begin{equation}
\label{eq:GFHlhs}
\abs{\int_{\Sph^2} f(\xb\cdot\zb)[P_\ell(\yb\cdot\zb)-P_\ell(\xb\cdot\yb)]\ud\Omega(\zb)}<+\infty.
\end{equation}
\end{lemma}
\begin{proof}
Without loss of generality, we assume that $\xb = (0,0,1)^\top$ is located at the North pole. 
If it were not, we could introduce an orthogonal rotation of coordinates to make it so. Then, let
\begin{equation}
\yb = (y_1,y_2,y_3)^\top\quad{\rm and}\quad \zb = (\sin\theta\cos\varphi,\sin\theta\sin\varphi,\cos\theta)^\top,
\end{equation}
be the other points on the sphere. We may write
\begin{equation}
\begin{array}{ll}
& \dsp \int_{\Sph^2} f(\xb\cdot\zb) [P_\ell(\yb\cdot\zb)-P_\ell(\xb\cdot\yb)]\ud\Omega(\zb),\\[10pt]
= & \dsp \int_{\Sph^2} (1-\xb\cdot\zb)f(\xb\cdot\zb)\dfrac{\yb\cdot\zb-\xb\cdot\yb}{1-\xb\cdot\zb}\dfrac{P_\ell(\yb\cdot\zb)-P_\ell(\xb\cdot\yb)}{\yb\cdot\zb-\xb\cdot\yb}\ud\Omega(\zb),
\end{array}
\end{equation}
provided we characterize the singularities that are introduced. The last term is bounded, for
\begin{equation}
\lim_{\yb\cdot\zb\to\xb\cdot\yb} \dfrac{P_\ell(\yb\cdot\zb)-P_\ell(\xb\cdot\yb)}{\yb\cdot\zb-\xb\cdot\yb} = P_\ell'(\xb\cdot\yb).
\end{equation}
Since it is a rational function with a single removable singularity, it has a finite spherical harmonic series,
\begin{equation}
\dfrac{P_\ell(\yb\cdot\zb)-P_\ell(\xb\cdot\yb)}{\yb\cdot\zb-\xb\cdot\yb} = \sum_{\lambda=0}^{\ell-1} \sum_{m=-\lambda}^{+\lambda} u_\lambda^m(\xb,\yb) Y_\lambda^m(\zb),
\end{equation}
where the coefficients are functions of $\xb$ and $\yb$. In terms of the spherical coordinates, the second term is
\begin{equation}
\dfrac{\yb\cdot\zb-\xb\cdot\yb}{1-\xb\cdot\zb} = \dfrac{(y_1\cos\varphi+y_2\sin\varphi)\sin\theta + y_3(\cos\theta-1)}{1-\cos\theta}.
\end{equation}
Then the integral is
\begin{equation}
\begin{array}{ll}
& \dsp \int_{\Sph^2} f(\xb\cdot\zb)[P_\ell(\yb\cdot\zb)-P_\ell(\xb\cdot\yb)]\ud\Omega(\zb),\\[10pt]
= & \dsp \lim_{\varepsilon\to0^+}\int_\varepsilon^\pi (1-\cos\theta)f(\cos\theta) \sin\theta\ud\theta\\[10pt]
& \dsp \times \int_0^{2\pi} \dfrac{(y_1\cos\varphi+y_2\sin\varphi)\sin\theta + y_3(\cos\theta-1)}{1-\cos\theta}\dfrac{P_\ell(\yb\cdot\zb)-P_\ell(\xb\cdot\yb)}{\yb\cdot\zb-\xb\cdot\yb}\ud\varphi,\\[10pt]
= & \dsp \sum_{\lambda=0}^{\ell-1} \sum_{m=-\lambda}^{+\lambda} u_\lambda^m(\xb,\yb) \lim_{\varepsilon\to0^+}\int_\varepsilon^\pi (1-\cos\theta)f(\cos\theta) \sin\theta\ud\theta\\[10pt]
& \dsp \times \int_0^{2\pi} \dfrac{(y_1\cos\varphi+y_2\sin\varphi)\sin\theta + y_3(\cos\theta-1)}{1-\cos\theta} Y_\lambda^m(\theta,\varphi)\ud\varphi.
\end{array}
\end{equation}
There are two cases to consider. On the one hand, if $m=0$, then $Y_\lambda^0(\theta,\varphi)$ is independent of $\varphi$, and
\begin{equation}
\begin{array}{ll}
& \dsp \int_0^{2\pi}\dfrac{(y_1\cos\varphi+y_2\sin\varphi)\sin\theta + y_3(\cos\theta-1)}{1-\cos\theta} Y_\lambda^0(\theta,\varphi)\ud\varphi,\\[10pt]
= & \dsp \int_0^{2\pi} \dfrac{y_3(\cos\theta-1)}{1-\cos\theta} Y_\lambda^0(\theta,\varphi)\ud\varphi = -2\pi y_3 Y_\lambda^0(\theta,\varphi).
\end{array}
\end{equation}
On the other hand, if $\abs{m}>0$, then the spherical harmonics contain at least a power of $\sin\theta$ such that
\begin{equation}
\dfrac{\sin\theta Y_\lambda^m(\theta,\varphi)}{1-\cos\theta}
\end{equation}
has a removable singularity as $\theta\to0^+$, proving that the integral in Equation~\eqref{eq:GFHlhs} is bounded.
\end{proof}

\begin{proof}[Proof of Proposition~\ref{proposition:GFH}]
First, let us note that by Lemma~\ref{lemma:GFHlhs} the left-hand side of \eqref{GeneralizedFunkHecke} is bounded. 
Let us now focus on the right-hand side of \eqref{GeneralizedFunkHecke} and consider the sequence $f_k$ of non-negative and \textit{integrable} functions 
defined by\footnote{This does not limit the proof to non-negative functions for we may split $f$ into $f^+ = \max\{f,0\}$ and $f^- = \min\{f,0\}$. However, it does simplify the argument.}
\begin{equation}
0 \leq f_k (t)= \left\{
\begin{array}{ll}
0, & t\in \{\vert f(t)\vert>k\}\cap\{t>1-1/k\},\\[10pt]
f(t), & \mbox{otherwise.}
\end{array}
\right.
\end{equation}
The generalized Funk--Hecke formula is clearly true for any $f_k$ as the Funk--Hecke formula is applicable in both instances with $P_\ell$ and $P_0=1$, respectively.
For any $\varepsilon>0$, let us write
\begin{equation}
\gamma_\ell = \int_{-1}^1[P_\ell(t)-1]f(t)\ud t = \int_{-1}^{1-\varepsilon}[P_\ell(t)-1]f(t)\ud t + \int_{1-\varepsilon}^{1}[P_\ell(t)-1]f(t)\ud t.
\label{eq1}
\end{equation}
\noindent Then there is a $K=K(\varepsilon)>\varepsilon^{-1}$ such that for any $k\geq K$, $f$ and $f_k$ coincide on $[-1,1-\varepsilon]$. Therefore
\begin{equation}
\gamma_\ell = \int_{-1}^{1-\varepsilon}[P_\ell(t)-1]f_k(t)\ud t + \int_{1-\varepsilon}^{1}[P_\ell(t)-1]f(t)\ud t,
\label{eq2}
\end{equation}
\noindent for some $k\geq K$. By adding and subtracting $\int_{1-\varepsilon}^{1}[P_\ell(t)-1]f_k(t)\ud t$ to \eqref{eq2} we obtain 
\begin{equation}
\gamma_\ell = \int_{-1}^{1}[P_\ell(t)-1]f_k(t)\ud t + \int_{1-\varepsilon}^{1}[P_\ell(t)-1]f(t)\ud t - \int_{1-\varepsilon}^{1}[P_\ell(t)-1]f_k(t)\ud t.
\label{eq3}
\end{equation}
\noindent The first term on the right-hand side of \eqref{eq3}, multiplied by $2\pi P_\ell(\xb\cdot\yb)$, can be rewritten as
\begin{equation}
 2\pi P_\ell(\xb\cdot\yb)\int_{-1}^{1}[P_\ell(t)-1]f_k(t)\ud t = \int_{\Sph^2} f_k(\xb\cdot\zb)[P_\ell(\yb\cdot\zb)-P_\ell(\xb\cdot\yb)]\ud\Omega(\zb),
\label{eq4}
\end{equation}
\noindent using the generalized Funk--Hecke formula for $f_k$.
Let us denote by $C_\varepsilon(\xb) = \{ \zb\in\Sph^2: 1-\varepsilon\le\zb\cdot\xb\le1\}$ the region on the sphere for which $f_k\ne f$.
We can rewrite \eqref{eq4} as
\begin{equation}
\begin{array}{ll}
&\dsp2\pi P_\ell(\xb\cdot\yb)\int_{-1}^{1}[P_\ell(t)-1]f_k(t)\ud t,\\[10pt]
=&\dsp\int_{\Sph^2\backslash C_\varepsilon(\xb)} f(\xb\cdot\zb)[P_\ell(\yb\cdot\zb)-P_\ell(\xb\cdot\yb)]\ud\Omega(\zb),\\[10pt]
=&\dsp\int_{\Sph^2} f(\xb\cdot\zb)[P_\ell(\yb\cdot\zb)-P_\ell(\xb\cdot\yb)]\ud\Omega(\zb)
-\int_{C_\varepsilon(\xb)} f(\xb\cdot\zb)[P_\ell(\yb\cdot\zb)-P_\ell(\xb\cdot\yb)]\ud\Omega(\zb).
\end{array}
\label{eq5}
\end{equation}
\noindent Combining \eqref{eq3}, multiplied by $2\pi P_\ell(\xb\cdot\yb)$, with \eqref{eq5} yields
\begin{equation}
\begin{array}{ll}
\dsp2\pi P_\ell(\xb\cdot\yb)\gamma_\ell = & \dsp\int_{\Sph^2} f(\xb\cdot\zb)[P_\ell(\yb\cdot\zb)-P_\ell(\xb\cdot\yb)]\ud\Omega(\zb)
-\int_{C_\varepsilon(\xb)} f(\xb\cdot\zb)[P_\ell(\yb\cdot\zb)-P_\ell(\xb\cdot\yb)]\ud\Omega(\zb)\\[10pt]
& \dsp+\,2\pi P_\ell(\xb\cdot\yb)\Bigg[\int_{1-\varepsilon}^{1}[P_\ell(t)-1]f(t)\ud t - \int_{1-\varepsilon}^{1}[P_\ell(t)-1]f_k(t)\ud t\Bigg].
\end{array}
\end{equation}
\noindent Using $0\leq f_k\leq f$ and $\vert P_\ell(\xb\cdot\yb)\vert\leq1$ leads to
\begin{equation}
\begin{array}{ll}
&\dsp \abs{\int_{\Sph^2} f(\xb\cdot\zb)[P_\ell(\yb\cdot\zb)-P_\ell(\xb\cdot\yb)]\ud\Omega(\zb) - 2\pi P_\ell(\xb\cdot\yb)\gamma_\ell},\\[10pt]
\leq & \dsp 4\pi\abs{\int_{1-\varepsilon}^{1}[P_\ell(t)-1] f(t)\ud t} + 
\abs{\int_{C_\varepsilon(\xb)} f(\xb\cdot\zb)[P_\ell(\yb\cdot\zb)-P_\ell(\xb\cdot\yb)]\ud\Omega(\zb)}.
\label{eq6}
\end{array}
\end{equation}
\noindent Both terms on the right-hand side of \eqref{eq6} go to zero as $\varepsilon\rightarrow0$ or equivalently as $k\to\infty$. Finally, the space
\begin{equation}
\left\{\hbox{Lebesgue measurable $f$} : t\mapsto[P_\ell(t)-1]f(t)\hbox{ absolutely integrable on $[-1,1]$, $\forall\ell=0,1,\ldots$}\right\},
\end{equation}
is the same as $t\mapsto(1-t)f(t)\in L^1([-1,1])$ since the only value of $t\in[-1,1]$ for which $P_\ell(t)-1=0$, $\forall\ell=0,1,\ldots,$ is $t=1$.
\end{proof}

\section{Proof of $\LL_\delta \rightarrow \LL_0$ strongly as $\delta\to0$}\label{appendix:strongconvergence}

Let
\begin{equation}
u(\theta,\varphi) = \sum_{\ell=0}^{+\infty}\sum_{m=-\ell}^{+\ell} u_\ell^mY_\ell^m(\theta,\varphi).
\end{equation}
\begin{definition}[see, e.g., \cite{atkinson2012}] For any $s\in\R$, the Sobolev space $H^s(\Sph^2)$ is the completion of $C^\infty(\Sph^2)$ with respect to the norm
\begin{equation}
\norm{u}_{H^s(\Sph^2)}^2 = \sum_{\ell=0}^{+\infty}\sum_{m=-\ell}^{+\ell}\Big(\ell+\tfrac{1}{2}\Big)^{2s}\abs{u_\ell^m}^2.
\end{equation}
\end{definition}
\begin{definition}
For every $\LL:H^s(\Sph^2)\to H^{s-2}(\Sph^2)$, let $\norm{\cdot}_s$ denote the induced operator norm defined by
\begin{equation}
\norm{\LL}_s = \sup_{0\ne u\in H^s(\Sph^2)}\dfrac{\norm{\LL u}_{H^{s-2}(\Sph^2)}}{\norm{u}_{H^s(\Sph^2)}}.
\end{equation}
\end{definition}
\begin{lemma}\label{lemma:boundspectrum}
For every $-1<\alpha<1$ and $0<\delta\le2$,
\begin{equation}
-\ell(\ell+1) \le \lambda_\delta(\ell) \le 0.
\end{equation}
\end{lemma}
\begin{proof}
The modulus of the spectrum of $\mathcal{L}_\delta$ is
\begin{equation}
\begin{array}{ll}
\abs{\lambda_\delta(\ell)} & \dsp = \abs{\frac{(1+\alpha)2^{2-\alpha}}{\delta^2}\int_{-1}^1 \dfrac{P_\ell\left[1 -\frac{\delta^2}{2}\left(\frac{1-x}{2}\right)\right]-1}{1-x}(1-x)^\alpha\ud x},\\[10pt]
& \dsp \le \frac{(1+\alpha)2^{2-\alpha}}{\delta^2}\int_{-1}^1 \abs{\dfrac{P_\ell\left[1 -\frac{\delta^2}{2}\left(\frac{1-x}{2}\right)\right]-1}{1-x}}(1-x)^\alpha\ud x.
\end{array}
\end{equation}
Using the inequalities
\begin{equation}
\begin{array}{ll}
\abs{P_\ell(x)-1} & \hspace{-0.25cm} \dsp = \abs{P_\ell(x)-P_{\ell-1}(x) + P_{\ell-1}(x)-1} \le \abs{P_\ell(x)-P_{\ell-1}(x)} + \abs{P_{\ell-1}(x)-1},\\[10pt]
& \hspace{-0.25cm} \dsp \le \sum_{\mu=1}^\ell \abs{P_\mu(x) - P_{\mu-1}(x)},
\end{array}
\end{equation}
and the identity~\cite[\S 18.6.1 \& 18.9.6]{olver2010}
\begin{equation}
P_\mu(x) - P_{\mu-1}(x) = -(1-x)P_{\mu-1}^{(1,0)}(x),
\end{equation}
where $P_{\mu-1}^{(1,0)}(x)$ is a Jacobi polynomial \cite[\S 18.3]{olver2010}, and since \cite[\S 18.14.1]{olver2010}
\begin{equation}
\sup_{x\in[-1,1]}\abs{P_{\mu-1}^{(1,0)}(x)} = \mu,
\end{equation}
we continue bounding the spectrum by
\begin{equation}
\begin{array}{ll}
\abs{\lambda_\delta(\ell)} & \hspace{-0.25cm}\dsp
\le \frac{(1+\alpha)2^{2-\alpha}}{\delta^2}\int_{-1}^1\left\{ \sum_{\mu=1}^\ell\abs{\frac{\delta^2}{2^2}P_{\mu-1}^{(1,0)}\left[1 -\frac{\delta^2}{2}\left(\frac{1-x}{2}\right)\right]} \right\} (1-x)^\alpha\ud x,\\[10pt]
&\hspace{-0.25cm}\dsp
= \frac{(1+\alpha)}{2^\alpha}\int_{-1}^1 \left\{ \sum_{\mu=1}^\ell \abs{P_{\mu-1}^{(1,0)}\left[1 -\frac{\delta^2}{2}\left(\frac{1-x}{2}\right)\right]} \right\} (1-x)^\alpha\ud x,\\[10pt]
& \hspace{-0.25cm}\dsp \le \frac{(1+\alpha)}{2^\alpha}\int_{-1}^1 (1-x)^\alpha\ud x \sum_{\mu=1}^\ell \mu,\\[10pt]
& \hspace{-0.25cm}\dsp = 2\sum_{\mu=1}^\ell \mu = \ell(\ell+1).
\end{array}
\end{equation}
Since the integrand in~\eqref{eigenvalues2} is non-positive, the spectrum is non-positive as well.
\end{proof}
\begin{theorem}
For every $-1<\alpha<1$, $0<\delta\le2$, and $s\ge0$, $\LL_\delta$ is a bounded operator from $H^s(\Sph^2)$ to $H^{s-2}(\Sph^2)$, and
\begin{equation}
\norm{\LL_\delta}_s \le \norm{\LL_0}_s < 1.
\end{equation}
\end{theorem}
\begin{proof}
These inequalities follow naturally from the definition of the induced operator norm, Lemma~\ref{lemma:boundspectrum}, and the completeness of spherical harmonics in $H^s(\Sph^2)\hookrightarrow L^2(\Sph^2)$.
\end{proof}
\begin{lemma}
For every $\ell=0,1,\ldots,$ $-1<\alpha<1$ and $\varepsilon>0$, there exists $0<\delta\le\frac{4\sqrt{\varepsilon}}{\ell+2}$ such that
\begin{equation}
\abs{\lambda_\delta(\ell) - \lambda_0(\ell)} \le \varepsilon\abs{\lambda_0(\ell)}.
\end{equation}
\end{lemma}
\begin{proof}
By the Lagrange form of the remainder in Taylor's theorem, there exists a $\xi\in[-1,1]$ such that
\begin{equation}
P_\ell\left[1 -\frac{\delta^2}{2}\left(\frac{1-x}{2}\right)\right] = P_\ell(1) - P_\ell'(1)\frac{\delta^2}{2^2}(1-x) + P_\ell''\left[1 -\frac{\delta^2}{2}\left(\frac{1-\xi}{2}\right)\right]\frac{\delta^4}{2^4}\frac{(1-x)^2}{2}.
\end{equation}
Rearranging terms,
\begin{equation}
\dfrac{P_\ell\left[1 -\frac{\delta^2}{2}\left(\frac{1-x}{2}\right)\right]-1}{1-x} = -\frac{\delta^2}{2^2}\dfrac{\ell(\ell+1)}{2} + P_\ell''\left[1 -\frac{\delta^2}{2}\left(\frac{1-\xi}{2}\right)\right]\frac{\delta^4}{2^4}\frac{(1-x)}{2}.
\end{equation}
Now, the second derivative of the Legendre polynomials is given in terms of the Jacobi polynomials \cite[\S 18.3]{olver2010}
\begin{equation}
P_\ell''(x) = \dfrac{(\ell+1)(\ell+2)}{2^2}P_{\ell-2}^{(2,2)}(x),
\end{equation}
and since \cite[\S 18.14.1]{olver2010}
\begin{equation}
\sup_{x\in[-1,1]}\abs{P_{\ell-2}^{(2,2)}(x)} = \dfrac{(\ell-1)\ell}{2},
\end{equation}
we obtain
\begin{equation}
\begin{array}{ll}
\abs{\lambda_\delta(\ell)-\lambda_0(\ell)} & \hspace{-0.25cm}\dsp \le \dfrac{(\alpha+1)}{2^\alpha}\int_{-1}^1\dfrac{(\ell-1)\ell(\ell+1)(\ell+2)\delta^2}{2^6}(1-x)^{\alpha+1}\ud x,\\[10pt]
& \hspace{-0.25cm}\dsp \le \dfrac{\alpha+1}{\alpha+2}\dfrac{(\ell-1)\ell(\ell+1)(\ell+2)\delta^2}{2^4},\\[10pt]
& \hspace{-0.25cm}\dsp \le \dfrac{(\ell-1)\ell(\ell+1)(\ell+2)\delta^2}{2^4},\\[10pt]
& \hspace{-0.25cm}\dsp \le \dfrac{\ell(\ell+1)(\ell+2)^2\delta^2}{2^4} = \varepsilon\abs{\lambda_0(\ell)}.
\end{array}
\end{equation}
\end{proof}
We are now in a position to prove strong convergence of the nonlocal operators to the Laplace--Beltrami operator.
\begin{theorem}
For every $-1<\alpha<1$ and $s\ge0$, $\LL_\delta \rightarrow \LL_0$ strongly as $\delta\to0$. That is,
\begin{equation}
\forall u\in H^s(\Sph^2),\quad\norm{(\LL_\delta - \LL_0)u}_{H^{s-2}(\Sph^2)} \to 0,\quad{\rm as}\quad\delta\to0.
\end{equation}
\end{theorem}
\begin{proof}
Let
\begin{equation}
u_n(\theta,\varphi) = \sum_{\ell=0}^{+n}\sum_{m=-\ell}^{+\ell} u_\ell^mY_\ell^m(\theta,\varphi),
\end{equation}
be the degree-$n$ truncation of $u\in H^s(\Sph^2)$. 
For every $\varepsilon>0$ and $n=0,1,\ldots,$ there exists $0<\delta\le\frac{4\sqrt{\varepsilon}}{n+2}$ such that
\begin{equation}
\begin{array}{ll}
\norm{(\LL_\delta - \LL_0)u}_{H^{s-2}(\Sph^2)} & \hspace{-0.25cm}\dsp 
\le \norm{(\LL_\delta - \LL_0)u_n}_{H^{s-2}(\Sph^2)} + \norm{\LL_\delta(u-u_n)}_{H^{s-2}(\Sph^2)} + \norm{\LL_0(u-u_n)}_{H^{s-2}(\Sph^2)},\\[10pt]
& \hspace{-0.25cm}\dsp \le \norm{(\LL_\delta - \LL_0)u_n}_{H^{s-2}(\Sph^2)} + \left(\norm{\LL_\delta}_s + \norm{\LL_0}_s\right)\norm{u-u_n}_{H^s(\Sph^2)},\\[10pt]
& \hspace{-0.25cm}\dsp < \varepsilon\norm{u}_{H^s(\Sph^2)} + 2\norm{u-u_n}_{H^s(\Sph^2)}.
\end{array}
\end{equation}
Thus, the norm is bounded by the first term that is arbitrarily small and the second term that converges to $0$ as $n\to\infty$ or equivalently as $\delta\to0$; for definiteness, choose
\begin{equation}
\varepsilon = \dfrac{1}{(n+2)^2},\quad \delta = \dfrac{4}{(n+2)^2},\quad\hbox{and take}\quad n\to\infty.
\end{equation}
Convergence of the second term follows from the completeness of spherical harmonics in $H^s(\Sph^2)\hookrightarrow L^2(\Sph^2)$.
\end{proof}

\end{document}